\documentclass[a4paper,11pt]{article}

\usepackage{geometry}
\usepackage{amssymb,amsmath,amsthm,amscd}
\usepackage[ansinew]{inputenc} 
\usepackage[english]{babel}

\usepackage{newtxtext}
\usepackage{newtxmath}
\usepackage{euscript}

\usepackage[T1]{fontenc}
\usepackage{mathrsfs}
\usepackage{amscd}
\usepackage{currfile}
\usepackage{hyperref}
\usepackage{bm}

\newtheorem{theorem}{Theorem}
\newtheorem*{theorem*}{Theorem}
\newtheorem{definition}[theorem]{Definition}

\newtheorem{proposition}[theorem]{Proposition}
\newtheorem{corollary}[theorem]{Corollary}
\newtheorem{example}[theorem]{Example}

\theoremstyle{definition}
\newtheorem{remark}[theorem]{Remark}

\newcommand{\hh}{{\mathbb{H}}}
\newcommand{\HH}{{\mathbb{H}}}

\newcommand{\cc}{{\mathbb{C}}}
\newcommand{\rr}{{\mathbb{R}}}

\newcommand{\nn}{{\mathbb{N}}}

\newcommand{\s}{{\mathbb{S}}}

\newcommand{\sr}{\mathcal{SR}}
\newcommand{\SF}{\mathcal{S}}

\newcommand{\A}{\mathcal{A}}
\newcommand{\I}{\mathcal{I}}
\newcommand{\J}{\mathcal{J}}
\newcommand{\E}{\mathcal{E}}

\newcommand{\SP}{\mathcal{S}}
\def\A{{A}}
\newcommand{\Pm}{\mathcal{P}}

\newcommand{\dibar}{\overline\partial}
\newcommand{\dcf}{\dibar}
\newcommand{\dif}{\vartheta}
\newcommand{\difbar}{\overline{\dif}}
\newcommand\thb[1]{\overline{\theta}_{#1}}
\newcommand\thet[1]{\theta_{#1}}
\newcommand\IM{\operatorname{Im}}
\newcommand\RE{\operatorname{Re}}

\newcommand\vs[1]{{#1}_s^\circ}
\newcommand\sd[1]{{#1}'_s}

\newcommand{\ui}{\imath}

\newcommand{\OO}{\Omega}

\newcommand{\mbb}{\mathbb}

\newcommand{\mr}{\mathrm}
\newcommand{\mscr}{\mathscr}
\newcommand{\R}{\mbb{R}}
\newcommand{\mc}{\mathcal}

\newcommand{\dd[3]}{\frac{\partial #2}{\partial #3}}
 \newcommand{\C}{\mbb{C}}

 \newcommand{\bc}{\begin{center}}
 \newcommand{\ec}{\end{center}}

\newcommand{\SD}{\EuScript{D}}
\newcommand{\be}{\begin{equation}}
\newcommand{\eeq}{\end{equation}}
\newcommand{\ba}{\begin{align*}}
\newcommand{\eal}{\end{align*}}

\begin{document}
\title{Wirtinger operators for functions of several quaternionic variables
}

\author{
\textsc{Alessandro Perotti}
\thanks{{\bf Statements and Declarations}: 
The author has no relevant financial or non-financial interests to disclose.
The author has no competing interests to declare that are relevant to the content of this article.
\\This work was supported by GNSAGA of INdAM (Istituto Nazionale di Alta Matematica), and by the grant ``Progetto di Ricerca INdAM, Teoria delle funzioni ipercomplesse e applicazioni''}
}
\date{
{\small
{Department of Mathematics, University of Trento, I--38123, Povo-Trento, Italy}\\
alessandro.perotti@unitn.it\\
ORCID: 0000-0002-4312-9504}
}

\maketitle

\begin{abstract}
We introduce Wirtinger operators for functions of several quaternionic variables. These operators are real linear partial differential operators which behave well on quaternionic polynomials, with properties analogous to the ones satisfied by the Wirtinger derivatives of several complex variables. Due to the non-commutativity of the variables, Wirtinger operators turn out to be of higher order, except the first ones that are of the first order. In spite of that, these operators commute with each other and satisfy a Leibniz rule for products. Moreover, they characterize the class of slice-regular polynomials, and more generally of slice-regular quaternionic functions.    
As a step towards the definition of the Wirtinger operators, we provide Almansi-type decompositions for slice functions and for slice-regular functions of several variables. We also  introduce some aspects of local slice analysis, based on the definition of locally slice-regular function in any open subset of the $n$-dimensional quaternionic space.
\medskip

\noindent \emph{2020 MSC:} Primary 30G35; Secondary 32A30, 32W50.

\medskip

\noindent \emph{Keywords:} Wirtinger operators, Quaternions, Slice-regular functions, Almansi decomposition.
\end{abstract}


\section{Introduction}

The main aim of this work is to introduce \emph{Wirtinger operators} for functions of several quaternionic variables. We are looking for partial differential operators which behave well when applied to quaternionic polynomials, in analogy with what happens for the Wirtinger derivatives $\dd{}{z_m}$, $\dd{}{\overline z_m}$ of several complex variables. 
We recall this behavior in the complex case. It holds 
\[
\textstyle\dd{z_j}{z_m}=\delta_{jm},\quad
\dd{\overline z_j}{z_m}=0,\quad
\dd{z_j}{\overline z_m}=0,\quad
\dd{\overline z_j}{\overline z_m}=\delta_{jm},
\]
and Leibniz rules for products are satisfied. To recall the relevance 
 of Wirtinger operators in the classical complex variables  case, we cite the words of Reinhold
Remmert \cite[p.\ 67]{Remmert}:  \emph{``the Wirtinger calculus \dots is quite indispensable in the function theory of several variables''}. This statement suggests that the investigation of similar operators for functions of $n$ quaternionic variables $x_1,\ldots,x_n$ should be of some interest. 

The differential operators  $\thet 1,\ldots\thet n, \thb1,\ldots,\thb n$ we are looking for should satisfy, in analogy with the complex case, the equalities
\[
\textstyle\thet m{(x_j)}=\delta_{jm},\quad
\thet m{(\overline x_j)}=0,\quad
\thb m{(x_j)}=0,\quad
\thb m{(\overline x_j)}=\delta_{jm}
\]
for all $1\le j,m\le n$. 
Moreover, they should commute with each other and respect a Leibniz type formula for products. The above properties are sufficient to compute the action of these operators on every quaternionic polynomial of the form 
$p(x)=\sum_{\ell,h}x_1^{\ell_1}\overline x_1^{h_1}\cdots x_n^{\ell_n}\overline x_n^{h_n}a_{\ell,h}$, where $\ell,h\in\mathbb N^n$ are multiindices. These polynomial functions are examples of (left) \emph{slice functions} of several quaternionic variables, while polynomials of the form $f(x)=\sum_{\ell}x_1^{\ell_1}\cdots x_n^{\ell_n}a_\ell$ are main examples of the subclass of \emph{slice-regular} quaternionic functions. In view of this, we will investigate with a particular emphasis the action of the Wirtinger operators on the class of slice functions of $n$ variables. 

Quaternionic \emph{slice-regular functions} in one variable were introduced in 2006-2007 by Gentili and Struppa \cite{GeSt2006CR,GeSt2007Adv}. 
The theory of slice-regular functions, also called \emph{slice analysis}, has then been  extended to the octonions,  to Clifford algebras, and more generally to real alternative *-algebras (see, e.g., \cite{CoSaSt2009Israel,GeStRocky,AIM2011}). See also \cite{GeStoSt2013,Struppa2015Algebras} for reviews of this function theory and extended references. 
Slice analysis has been generalized to several variables in various contexts \cite{GhPe_ICNPAA,MoscowSeveral,SeveralA}.  See also \cite{Struppa2015Several} for a review of this theory and \cite{AngellaBisi,GentiliGoriSarfatti} for applications of slice analysis of several quaternionic variables to quaternionic manifolds.
Our approach to slice analysis in several variables, described in \cite{GhPe_ICNPAA,MoscowSeveral,SeveralA}, is based on the concept of stem functions of several variables and on the introduction  of a family of commuting complex structures on the real vector space $\R^{2^n}$. Several quaternionic variables have been investigated also by Colombo, Sabadini and Struppa \cite{CoSaSt2013Indiana}. Their approach via stem functions is similar to ours, but the definition of regularity is different, as shown in \cite[\S3]{SeveralA}. 
In the case of two quaternionic variables, an approach equivalent to  \cite{GhPe_ICNPAA,MoscowSeveral,SeveralA} has been described by Dou, Ren, I. Sabadini and Wang in \cite{DouRenSabadiniWang}. 
A slice theory of several variables has been proposed also by Ren and Yang \cite{Ren_yang_2020} for octonions and then extended to other algebras $A$ in \cite{DouRenSabadiniAMPA}. The major difference with our theory is that the authors define slice functions 
on a class of non-open (in the euclidean topology) subsets of the space $A^n$, where the variables associate and commute. 
A different approach to several quaternionic variables, based on fiber bundles, has been given by Gonz\'alez-Cervantes in the recent paper \cite{GonzalesAACA}.

In one variable, the quaternionic Wirtinger operators $\thet1$, $\thb1$ are already known (even if not denoted with this name). They are the global differential operators $\dif$, $\difbar$ introduced in \cite{Gh_Pe_GlobDiff} (see definition \eqref{def:theta} in \S\ref{sec:1var}).  
Since 
the operators $\dif$, $\difbar$ coincide  on slice functions with the \emph{slice derivatives} $\dd{}{x}$, $\dd{\ }{x^c}$ respectively (see \S\ref{sec:1var} for definitions), they satisfy the required formulas when applied to powers $x^k$ and $\overline x^k$. 

In several variables, one can still consider the first-order differential operators $\dif_{x_m}$ and $\difbar_{x_m}$ defined in the same way (see formula \eqref{def:thb} in \S\ref{sec:sev_var}) with respect to the variable $x_m$. However, while the first operators $\dif_{x_1}$ and $\difbar_{x_1}$ still behave well, the other operators $\dif_{x_2},\ldots,\dif_{x_n}$ and $\difbar_{x_2},\ldots,\difbar_{x_n}$ do not, due to the non-commutativity of the variables. Observe that one can define the \emph{slice partial derivatives} also for slice functions of several variables (see \S\ref{sec:sev_var}). These operators have the correct algebraic properties but, in spite of their name, they are not defined as differential operators. 
To define the correct Wirtinger operators (Definition \ref{def:Wirtinger}), we make use of the spherical Dirac operators $\Gamma_{x_m}$ together with the operators $\dif_{x_m}$ and $\difbar_{x_m}$.
It turns out that the Wirtinger operators $\thet m$, $\thb m$ are $\R$-linear partial differential operators of the first order in every single variable $x_1,\ldots x_m$, but of total order $m$. Their action on slice functions coincide with that of slice partial derivatives. In particular, the common kernel of the operators $\thb1,\ldots,\thb n$ is the class of slice-regular functions of $n$ variables (Theorem \ref{teo:Wirtinger_regular}). 

A second aim of this work 
is to introduce some aspects of \emph{local slice analysis} in several variables.
The definitions of slice-regular function, both in the original approach \cite{GeSt2007Adv}, with functions defined on \emph{slice domains} intersecting the real axis, and in the \emph{stem function} approach \cite{AIM2011}, where functions are defined on domains that are axially symmetric around the real axis, are not of local character. 

We extend to several variables what obtained in \cite{CR_S_operators} for one variable, making a systematic use of the one-variable characterization of sliceness and of slice-regularity in several variables \cite[Proposition 2.23 and Theorem 3.23]{SeveralA}.
Using the relations with the Cauchy-Riemann-Fueter operators $\dibar_{x_m}$, we are able to refine the original definition of slice-regularity and find a formulation that has a natural local version. We propose a definition of \emph{locally slice-regular function} in any open subset of $\hh^n$ (Definition\ \ref{def:loc_several}), compatible with the existing definitions. Using the spherical Dirac operators $\Gamma_{x_m}$ in place of $\dibar_{x_m}$, we also refine the definition of slice function to get the concept of \emph{locally strongly slice function} of $n$ variables (Definition\ \ref{def:loc_slice_several}). Among this class of functions, Wirtinger operators still characterize the subclass of locally slice-regular functions (Theorem \ref{teo:Wirtinger_regular_local}). 

We describe in more detail the structure of the paper. In section \ref{sec:Preliminaries}, we recall some basic definitions about quaternionic slice analysis. Subsection \ref{sec:1var} presents the basic definitions for the one-variable theory, while subsection \ref{sec:sev_var} gives a rapid introduction to the (less known) several variables theory (see \cite{SeveralA} for full details in the more general context of real alternative *-algebras). Here we also recall the one-variable characterization of sliceness and of slice-regularity in several variables.  In Section \ref{sec:AlmansiDecomposition}, we prove some Almansi-type decompositions for slice functions of several quaternionic
variables. 
The slice-regular version of this result (Theorem \ref{teo:Almansi_several}) 
has been proved 
in \cite{Binosi1}. It expresses, for any $m=1,\ldots, n$, the slice-regular function $f$ as a combination of $2^m$ functions $\left\{\SP^m_K(f)\right\}_{K\in\Pm(m)}$ which exhibit specific harmonicity and regularity properties (see formula \eqref{def:sp}). In the slice-regular case, the functions $\SP^m_K(f)$ can be replaced by functions $\A^m_K(f)$ obtained by iterated application of the Cauchy-Riemann-Fueter operators $\dibar_{x_m}$ (see formula \eqref{eq:amk}). However, the functions $\SP^m_K(f)$ can indeed be defined for any slice function of $n$ variables as iterated one-variable spherical derivatives. In Section \ref{sec:AlmansiDecomposition2} we show that the same decomposition holds for any slice function (Theorem \ref{teo:slice_Almansi_several}). As expected, in this case the harmonicity and regularity properties of the functions $\SP^m_K(f)$ are lost. 
The aim of Section \ref{sec:Locally_slice-regular} is to introduce a local slice analysis in several quaternionic variables. It presents the definitions of \emph{strongly slice-regular} and \emph{locally strongly slice-regular} function on any open subset of $\hh^n$ and the proofs of two fundamental extendibility theorems: a global one for strongly slice-regular functions and the corresponding local version for locally strongly slice-regular functions. 
In Section \ref{sec:Wirtinger operators}, we give the definition of the Wirtinger operators (Definition \ref{def:Wirtinger}) as differential operators acting on any sufficiently smooth function on open subsets of $\hh^n$. This definition requires to associate to $f$ another set $\left\{\Gamma^m_K(f)\right\}_{K\in\Pm(m)}$ of $2^m$ functions obtained by iterated application of the spherical Dirac operators $\Gamma_{x_m}$ (see formula \eqref{eq:Gamma_mk}). On slice functions, the action of the differential operators $\Gamma^m_K$ coincides with the iterated spherical derivatives  $\SP^m_K(f)$. This fundamental fact and the Almansi-type decompositions allows to obtain the main properties of the Wirtinger operators (Proposition \ref{pro:powers}, Theorems \ref{teo:Wirtinger} and \ref{teo:Wirtinger_regular}, Proposition \ref{pro:Wirtinger_conjugate}). 
In Section \ref{sec:Local_slice_functions}, after extending the Almansi decomposition of Theorem \ref{teo:slice_Almansi_several} to any sufficiently smooth (not necessarily slice) function using the $\Gamma^m_K(f)$'s, we give the definition of \emph{strongly slice} and \emph{locally strongly slice} function (Definition \ref{def:loc_slice_several}). Then we prove an extendibility theorem (Theorem \ref{teo:slice_extension}) and the characterization of slice-regularity in the class of (locally) strongly slice functions, by means of the Wirtinger operators $\thb1,\ldots,\thb n$ (Theorem \ref{teo:Wirtinger_regular_local}).

\section{Preliminaries}\label{sec:Preliminaries}

\subsection{The one-variable slice function theory}\label{sec:1var}

Slice function theory is based on the ``slice'' decomposition of the quaternionic space $\hh$ in complex planes. For each imaginary unit $J$ in the sphere
 \[\s=\{J\in\HH\ |\ J^2=-1\}=\{x_1i+x_2j+x_3k\in\HH\ |\ x_1^2+x_2^2+x_3^2=1\},\]
let $\C_J=\langle 1,J\rangle\simeq\C$ denote the subalgebra generated by $J$. Then it holds
\[\HH=\bigcup_{J\in \s}\C_J, \quad\text{with $\C_J\cap\C_K=\R$\ for every $J,K\in\s,\ J\ne\pm K$.}\]
Quaternionic slice functions are functions which are compatible with the slice character of $\hh$.
More precisely, let $D$ be a subset of $\C$ that is invariant w.r.t.\ complex conjugation. 
Let $\hh\otimes_{\R}\C$ be the complexified algebra, whose elements $w$ are of the form $w=a+\ui b$ with $a,b\in \hh$ and $\ui^2=-1$. In $\hh\otimes_{\R}\C$ we consider the complex conjugation mapping $w=a+\ui b$ to $\overline w=a-\ui b$  for all $a,b\in \hh$.
If a function $F: D \to \hh\otimes_{\R}\C$ satisfies  $F(\overline z)=\overline{F(z)}$ for every $z\in D$, then $F$  is called a \emph{stem function} on $D$. For every $J\in\s$, consider the real $^*$-algebra isomorphism $\phi_J:\C\to\C_J$ defined by $\phi_J(\alpha+i\beta):=\alpha+J\beta$ for all $\alpha,\beta\in\R$. 
Let $\OO_D$ be the \emph{axially symmetric} (or \emph{circular}) subset of the quadratic cone defined by 
\[
\OO_D=\bigcup_{J\in\s}\phi_J(D)=\{\alpha+J\beta\in\hh : \alpha,\beta\in\R, \alpha+\ui\beta\in D,J\in\s\}.
\]
The stem function $F=F_1+\ui F_2:D \to \hh\otimes_{\R}\C$  induces the \emph{(left) slice function} $f=\I(F):\OO_D \to\hh$ in the following way: if $x=\alpha+J\beta =\phi_J(z)\in \OO_D\cap \C_J$, then  
\[ f(x)=F_1(z)+JF_2(z),\]
where $z=\alpha+i\beta$. 

Suppose that $D$ is open. The slice function $f=\I(F):\OO_D \to \hh$ is called \emph{(left) slice regular} if $F$ is holomorphic w.r.t.\ the complex structure on $\hh\otimes_{\R}\C$ defined by left multiplication by $\ui$. 
For example, polynomial functions $f(x)=\sum_{j=0}^d x^ja_j$ with right quaternionic coefficients are slice regular on $\hh$. If a domain $\OO$ in $\hh$ is axially symmetric and intersects the real axis, then this definition of slice regularity is equivalent to the one proposed by Gentili and Struppa \cite{GeSt2007Adv}. We will denote by $\sr(\OO)$ the right quaternionic module of slice-regular functions on $\OO$, and by $\SF(\OO)$ the module of slice functions on $\OO$.

The pointwise product of two stem functions $F$, $G$ on $D$ defines the \emph{slice product} of the slice functions $f=\I(F)$, $g=\I(G)$ on $\OO$, i.e., $f\cdot g=\I(FG)$.
The slice product has an interpretation in terms of pointwise quaternionic product:
\[
(f\cdot g)(x)=
\begin{cases}f(x)g(f(x)^{-1}xf(x))&\text{\ if $f(x)\ne0$,}\\0&\text{\ if $f(x)=0$}.
\end{cases}
\]
If $f,g\in\sr(\OO)$, then also their slice product $f\cdot g$ is slice-regular on $\OO$. The function $f=\I(F)$ is called \emph{slice-preserving} if the $\hh$-components $F_1$ and $F_2$ of the stem function $F$ are real-valued.
Under this condition the slice product $f\cdot g$ coincides with the pointwise product of $f$ and $g$. The same holds if $F$ and $G$ are $\hh$-valued. In this case, we will denote it simply by $fg$. 

To any slice function $f=\I(F):\OO_D \to \hh$, one can associate the function $\vs f:\OO_D \to \hh$, called \emph{spherical value} of $f$, and the function $f'_s:\OO_D \setminus \R \to \hh$, called  \emph{spherical derivative} of $f$, defined as
\[
\vs f(x):=\tfrac{1}{2}(f(x)+f(\overline x))
\quad \text{and} \quad
f'_s(x):=\tfrac{1}{2}\IM(x)^{-1}(f(x)-f(\overline x)).
\] 
$\vs f$ and $f'_s$ are slice functions, constant on every set $\s_x:=\alpha+\beta\s$, satisfying the formula
\[
f(x)=\vs f(x)+\IM(x)f'_s(x)
\]
for all $x\in\OO_D\setminus \R$ and the product rule $\sd{(f\cdot g)}=\sd f \vs g+\vs f\sd g$.  A remarkable property of the spherical derivative of a slice-regular function is the harmonicity of its real components w.r.t.\ the four real variables, while the slice-regular function itself is only biharmonic \cite{Harmonicity}.

The \emph{slice derivatives} 
(sometimes called \emph{complex derivatives}) of a slice function $f=\I(F)$ are defined by means of the  Cauchy-Riemann operators applied to the inducing stem function $F$:
\[\dd{f}{x}=\I\left(\dd{F}{z}\right),\quad \dd{f\;}{x^c}=\I\left(\dd{F}{\overline z}\right).\]
Observe that $f\in\sr(\OO)$ if and only if $\dd{f\;}{x^c}=0$ and if $f\in\sr(\OO)$ then also $\dd{f}{x}$ is slice-regular on $\OO$. Moreover, the slice derivatives satisfy the Leibniz product rule w.r.t.\ the slice product.

We recall from \cite{Gh_Pe_GlobDiff} the definition of the global differential operator $\difbar:\mathcal{C}^1(\OO\setminus \R,\hh) \to \mathcal{C}^0(\OO\setminus \R, \hh)$ associated with quaternionic slice-regular functions and its conjugate $\dif$:
\be\label{def:theta}
\dif=\frac12\left(\dd{}{x_0}+(\IM(x))^{-1} \sum_{i=1}^3x_{i} \, \dd{}{x_{i}}\right),\quad
\difbar=\frac12\left(\dd{}{x_0}-(\IM(x))^{-1} \sum_{i=1}^3x_{i} \, \dd{}{x_{i}}\right).
\eeq
On the slice $\C_I$ the operator $\difbar$ coincides with the standard Cauchy-Riemann operator of $\C_I$. 
For every slice function $f$ on $\OO$, it holds $\dif f=\dd{f}{x}$ and $\difbar f =\dd f{x^c}$  on $\OO\setminus\R$ (see \cite[Theorem 2.2]{Gh_Pe_GlobDiff}).

\subsection{The several variables slice function theory}\label{sec:sev_var}

We recall some basic definitions from \cite{GhPe_ICNPAA,SeveralA}.
Let $D$ be a subset of $\C^n$ that is invariant under complex conjugations: $\overline{z}^h:=(z_1,\ldots,\overline z_h,\ldots,z_n) \in D$ for all $z \in D$ and for all $h \in \{1,\ldots,n\}$.
Let $\{e_K\}_{K \in \mc{P}(n)}$ be a fixed basis of the real vector space $\R^{2^n}$. We identify $\R$ with the real vector subspace of $\R^{2^n}$ generated by $e_\emptyset\in\R^{2^n}$, and we write $e_{\emptyset}=1$. For simplicity, we set $e_k:=e_{\{k\}}$ for all $k\in\{1,\ldots,n\}$.

Each element $x$ of the tensor product $\hh \otimes \R^{2^n}$ can be uniquely written as $x=\sum_{K\in\Pm(n)}e_Ka_K$ with $a_K\in \hh$. Given any function $F:D \to \hh \otimes \R^{2^n}$, there exist unique functions $F_K:D\to \hh$ such that $F=\sum_{K \in \Pm(n)}e_KF_K$ ($F_K$ is called the \emph{$K$-component of $F$}).
A function $F:D \to \hh \otimes \R^{2^n}$ with $F=\sum_{K \in \Pm(n)}e_KF_K$ is a \emph{stem function} if 
\begin{equation} \label{eq:stem}
F_K(\overline{z}^h)=
\begin{cases}
F_K(z)  & \text{ if  $\,h \not\in K$,}
\\
-F_K(z)  & \text{ if $\,h \in K$}
\end{cases}
\end{equation}
for all  $z \in D$, $K \in \Pm(n)$ and $h \in \{1,\ldots,n\}$. 
Let $\OO_D$ be the \emph{axially symmetric} (or \emph{circular}) open subset of $\HH^n$ associated to  $D$, defined as
\[
\OO_D:=\{(\alpha_1+J_1\beta_1,\ldots\alpha_n+J_n\beta_n)\in \HH^n : J_1,\ldots,J_n\in\s, (\alpha_1+i\beta_1,\ldots,\alpha_n+i\beta_n)\in D\}.
\]
The \emph{(left) slice function} $f=\I(F):\OO_D \rightarrow \HH$ induced by $F$ is the function obtained by setting, for each $x=(x_1,\ldots,x_n)=(\alpha_1+J_1\beta_1,\ldots,\alpha_n+J_n\beta_n)\in\OO_D$,
\[
f(x)=\sum_{K \in \Pm(n)}J_KF_K(z)
\]
where $J_K=J_{k_1}\cdots J_{k_p}$ if $K=\{k_1,\ldots,k_p\}\in\Pm(n)\setminus\{\emptyset\}$ with $k_1<\cdots<k_p$,  $J_\emptyset=1$, and $z=(z_1,\ldots,z_n)=(\alpha_1+i\beta_1,\ldots,\alpha_n+i\beta_n)\in D$. 

We will denote by $\SF^0(\OO)$ the right $\hh$-module of slice functions in $\OO=\OO_D$ induced by continuous stem functions $F\in \mathcal C^0(D)$, and by $\SF^1(\OO)$ the submodule of slice functions in $\OO=\OO_D$ induced by stem functions $F\in\mathcal C^1(D)$.

Let $\mathcal J_1,\ldots\mathcal J_n$ be the commuting complex structures on $\R^{2n}\simeq\C^{\otimes n}$ induced, respectively, by the standard structures of the $n$ copies of $\C$. 
The isomorphism $\R^{2n}\simeq\C^{\otimes n}$ maps the basis element $e_K$ of $\R^{2n}$ to the element $v_K$ of $\C^{\otimes n}$ defined as
\[
v_K=v_1\otimes\cdots\otimes v_n \text{\quad with $v_h=i$ if $h\in K$, $v_h=1$ if $h\not\in K$}.
\]
Explicitly, the complex structures are defined by
\begin{equation}
\mathcal J_h(e_K)=
\begin{cases}
e_{K \cup \{h\}} & \text{\quad if }h\not\in K,\\
-e_{K \setminus \{h\}} & \text{\quad if }h\in K.
\end{cases}
\end{equation}
In particular, it holds $\mathcal J_h(e_h)=-1$ for $h=1,\ldots,n$. We extend these structures to $\HH\otimes\R^{2n}$ by setting $\mathcal J_h(a\otimes v)=a\otimes\mathcal J_h(v)$ for all $a\in\HH$ and $v\in\R^{2n}$.

Let $F:D \to \HH \otimes \R^{2n}$ be a stem function of class $\mscr{C}^1$. For each $h=1,\ldots, n$, we denote by $\partial_h$ and $\dibar_h$ the \emph{Cauchy-Riemann operators} w.r.t.\ the standard complex structure on $D$ and  $\J_h$ on $\HH\otimes \R^{2n}$, i.e.
\be\label{def:CR}
\partial_hF=\frac{1}{2}\left(\dd{F}{\alpha_h}-\J_h\left(\dd{F}{\beta_h}\right)\right)
\;\;\text{ and }\;\;
\dibar_hF=\frac{1}{2}\left(\dd{F}{\alpha_h}+\J_h\left(\dd{F}{\beta_h}\right)\right),
\eeq
where $\alpha_h+i \beta_h:D\to\C$ is the $h^{\mr{th}}$-coordinate function of $D$. Let $f=\I(F):\OO_D\to\HH$ and let $h\in\{1,\ldots,n\}$. Each operator of the type $\partial_h$ or $\dibar_h$ commutes with each other. We define the \emph{slice partial derivatives} of $f$ as the following slice functions on $\OO_D$:
\be\label{def:spd}
\dd{f}{x_h}:=\I(\partial_hF)
\;\;\text{ and }\;\;
\dd{f}{x_h^c}:=\I(\dibar_hF).
\eeq
The slice function $f=\I(F)$ is called \emph{slice-regular} on $\OO_D$ if $F$ is holomorphic w.r.t.\ $\J_1,\ldots,\J_n$, i.e., $\dibar_hF=0$ for $h\in\{1,\ldots,n\}$. Equivalently, $\dd f{x_h^c}=0$ for every $h$. For example, every polynomial function $p(x)=\sum_{\ell}x_1^{\ell_1}\cdots x_n^{\ell_n}a_\ell$ with ordered variables and right quaternionic coefficients $a_\ell$, with $\ell=(\ell_1,\ldots,\ell_n)$, is slice-regular on $\hh^n$. We will denote by $\sr(\OO)$ the right quaternionic module of slice-regular functions on $\OO=\OO_D$. 

Every product on $\HH\otimes\R^{2n}$ induces a product on stem functions, and hence a structure of real algebra on the set of slice functions. In the following we take the product obtained identifying as above the real algebra $\R^{2n}$ with $\C^{\otimes n}$. 
The corresponding (tensor) product in $\hh\otimes\R^{2n}$ is the linear extension of
\[
(a\otimes e_H)(b\otimes e_K)=(ab)\otimes (e_H e_K)=(ab)\otimes(-1)^{|H\cap K|}e_{H\Delta K}.
\]
Let $f=\I(F),g=\I(G):\OO_D \to \HH$ be slice functions. We define the \emph{{slice product} $f\cdot g:\OO_D\to \HH$} of $f$ and $g$ by $f\cdot g:=\I(FG)$, where $FG$ is the pointwise product defined by $(FG)(z)=F(z)G(z)$ in $\HH\otimes\R^{2n}$ for all $z\in D$. 

A slice function $g\in\SF(\OO)$ is called \emph{circular} if $g=\I(G)$ with $G$ an $\hh$-valued stem function. 
If $f,g\in\SF(\OO)$, with $g$ circular, then $f\cdot g$ coincides with the quaternionic pointwise product $fg$ (see \cite[Lemma~2.51]{SeveralA}). Observe that circular slice functions on $\OO$ are constant on sets of the form   $\s_{x_1}\times\cdots\times\s_{x_{n}}$ for every $x\in\OO$. 

The slice partial derivatives satisfy Leibniz's rule w.r.t.\ the slice product: for each slice functions $f,g\in\mc{S}^1(\OO_D)$ and $h=1,\ldots,n$, it holds
\[
\frac{\partial}{\partial x_h}(f\cdot g)=\frac{\partial f}{\partial x_h}\cdot g+f\cdot \frac{\partial g}{\partial x_h}\text{\quad and\quad}
\frac{\partial}{\partial x_h^c}(f\cdot g)=\frac{\partial f}{\partial x_h^c}\cdot g+f\cdot \frac{\partial g}{\partial x_h^c}.
\]
In particular, the slice product preserves slice regularity. 

In the following we will denote by $\dif_{x_m}$ and $\difbar_{x_m}$ the global differential operators  w.r.t.\ the quaternionic variable $x_m=x_{m_0}+ix_{m_1}+jx_{m_2}+kx_{m_3}$, defined as in \eqref{def:theta}: 
\be\label{def:thb}
\dif_{x_m}=\frac12\left(\dd{}{x_{m_0}}+(\IM(x_m))^{-1}\mathbb E_{x_m}\right),\quad
\difbar_{x_m}=\frac12\left(\dd{}{x_{m_0}}-(\IM(x_m))^{-1}\mathbb E_{x_m}\right),
\eeq
where $\mathbb E_{x_m}=\sum_{i=1}^3x_{m_i}\dd{\ }{x_{m_i}}$ is the Euler operator w.r.t.\ the vector variable $(x_{m_1},x_{m_2},x_{m_3})\in\R^3$.

\subsection*{One-variable characterization}
The concepts of spherical value and spherical derivative in one variable have a central role in the characterization of slice regularity in several variables in terms of separate one-variable regularity. 
Assume that $g$ is a slice function w.r.t.\ $x_h$ and define the functions $\SD_{x_h}^0g(x)$ and $\SD_{x_h}^1g(x)$ obtained taking the spherical value and the spherical derivative of $g$ w.r.t.\ $x_h$: 
\begin{equation}\label{eq:sd}
\SD_{x_h}^0g:=(g)^\circ_{s,x_h} \text{\quad and\quad}
\SD_{x_h}^1g:=(g)'_{s,x_h}.
\end{equation}
Let $f\in\SF(\OO)$ be a slice function on $\OO\subseteq\hh^n$, let $K\in\Pm(n)$ and let $\epsilon=\mathbf{1}_{K}$ be the characteristic function of $K$. Then $f$ is a slice function w.r.t.\ $x_1$ and, for each $h\in\{2,\ldots,n\}$, the \emph{truncated spherical $\epsilon$-derivative} $\SD_\epsilon f:=\SD^{\epsilon(h-1)}_{x_{h-1}}\cdots \SD^{\epsilon(1)}_{x_1}f$, obtained iterating \eqref{eq:sd}, is a well-defined slice function w.r.t.\ $x_h$. 
Moreover, $f\in\sr(\OO)$ if and only if $f$ is slice-regular w.r.t.\ $x_1$ and, for each $h\in\{2,\ldots,n\}$ and $K\in\Pm(n)$, $\SD_\epsilon f$ is slice-regular w.r.t.\ $x_h$ \cite[Proposition 2.23 and Theorem 3.23]{SeveralA}.  

For example, when $n=2$, a slice function $f$ is slice-regular in $x=(x_1,x_2)$ if and only if $f$ is slice-regular w.r.t.\ $x_1$, and the spherical value and spherical derivative of $f$ w.r.t.\ $x_1$ are slice-regular w.r.t.\ $x_2$.

\section{An Almansi-type decomposition for slice functions of several quaternionic variables}\label{sec:AlmansiDecomposition}

\subsection{Decomposition for slice-regular functions of several quaternionic variables}\label{sec:AlmansiDecomposition1}

In this section we recall and extend some results from \cite{Binosi1}, where the Almansi-type decomposition obtained in \cite{AlmansiH} was generalized to several variables. We first recall the one-variable result. 
Let $\dcf$ denote the \emph{Cauchy-Riemann-Fueter operator} on $\hh$:
\[\dcf =\frac12\left(\dd{}{x_0}+i\dd{}{x_1}+j\dd{}{x_2}+k\dd{}{x_3}\right).\]

\begin{theorem}[Almansi Theorem in one quaternionic variable \cite{AlmansiH}]
\label{teo:Almansi}
Let $f$ be slice-regular in an axially symmetric set $\OO\subseteq\hh$. Then there exist two unique quaternionic-valued zonal harmonic functions $h_1$, $h_2$ with pole $1$, such that 
\[
f(x)=h_1(x)-\overline x h_2(x)\text{\quad  $\forall x\in\OO$.}
\]
Moreover, it holds $h_1= (xf)'_s=-\dcf(xf)$, $h_2=f'_s=-\dcf f$ and $f$ is slice-preserving if and only if $h_1$ and $h_2$ are real-valued.
\end{theorem}

We recall that a \emph{zonal harmonic functions with pole $1$} on $\OO$ is a harmonic function that is constant on the spheres $\s_y\cap \OO$ for every $y\in \OO$.

\begin{remark}\label{rem:sd}
If $f$ is any slice function in an axially symmetric set $\OO\subseteq\hh$, and $f(x)=h_1(x)-\overline x h_2(x)$, with $h_1$ and $h_2$ constant on every sphere $\s_x\subseteq\OO$, then a direct computation yields $h_1= (xf)'_s$, $h_2=f'_s$.
\end{remark}

The decomposition given by Theorem~\ref{teo:Almansi} is also called \emph{zonal decomposition} of $f$. Using the previous theorem and the one-variable interpretation of slice regularity (see \S\ref{sec:Preliminaries} and \cite[\S3.4]{SeveralA}), we will see below that the result can be extended to several variables. 

Let $\Pm(m)$ denote the set of all subsets of $\{1,\ldots,m\}$. 
Let $f\in\SF(\OO)$ be a slice function in an axially symmetric open set $\OO=\OO_D\subseteq\hh^n$.  Define recursively 
the slice functions $\SP^m_K(f)$, for every $m\in\{0,\ldots, n\}$ and every $K\in\Pm(m)$ as follows:
\begin{equation}\label{def:sp}
\begin{cases}
{}\SP^0_\emptyset(f):=f&\text{\quad for }m=0,\\
\SP^m_K(f):=\left(x_m^{\mathbf{1}_K(m)}\SP^{m-1}_{K\setminus\{m\}}(f)\right)'_{s,x_m} &\text{\quad for }m\in\{1,\ldots,n\}.
\end{cases}
\end{equation}
where $\mathbf{1}_{K}$ is the 
characteristic function of $K$ and $(g)'_{s,x_m}$ denotes the spherical derivative of $g$ with respect to the variable $x_m$. 

As a consequence of the identity $(g)^\circ_{s,x_m}=(x_m g)'_{s,x_m}-\RE(x_m) (g)'_{s,x_m}$, valid for every slice function $g$ w.r.t.\ $x_m$ and every $m\in\{1,\ldots,n\}$, it follows that the functions $\SP^m_K(f)$ can be written as combinations of the \emph{truncated spherical $\epsilon$-derivatives} $\SD_\epsilon f$ of $f$ (see \cite[Definition~2.24]{SeveralA} and \S\ref{sec:Preliminaries}) and vice versa. 
From \cite[Proposition~2.23]{SeveralA} it follows that every slice function $f\in\SF(\OO)$ is a slice function w.r.t.\ $x_1$ and the functions $\SP^{m-1}_{K\setminus\{m\}}(f)$ and $x_m\SP^{m-1}_{K\setminus\{m\}}(f)$ are slice functions w.r.t.\ $x_m$. Therefore the definition given in \eqref{def:sp} is well-posed.

The functions $\SP^m_K(f)$ are defined on the dense subset $\OO\setminus\R_m$ of $\OO$, where 
\[
\R_m:=\textstyle\bigcup_{k=1}^m\{(x_1,\ldots,x_n)\in\hh^n\,|\,x_k\in\R\}.
\]
The common domain of definition of the functions $\SP^m_K(f)$ is the dense subset $\OO\setminus\R_\bullet$ of $\OO$, where $\R_\bullet=\R_n=\bigcup_{k=1}^n\{(x_1,\ldots,x_n)\in \hh^n\,|\,x_k\in\R\}$.

In the following, given $H=\{h_1,\ldots,h_p\}\in\Pm(n)$ with $h_1<\cdots <h_p$, we will denote by $J_H$ the product  $J_{h_1}\cdots J_{h_p}$.

\begin{proposition}\label{pro:SP_slice}
Let $f\in\SF(\OO)$, with $\OO=\OO_D$ an axially symmetric open set in $\hh^n$. 
For every $m\in\{0,\ldots,n\}$ and $K\in\Pm(m)$, the functions $\SP^m_K(f)$ are slice functions of $n$ variables in $\OO\setminus\R_\bullet$. More precisely, $\SP^m_K(f)$ is a slice function in $\OO\setminus\R_m\supseteq\OO\setminus\R_\bullet$. Moreover, for every $y\in\OO\setminus\R_m$, $\SP^m_K(f)$ is constant on the set $(\s_{y_1}\times\cdots\times\s_{y_{m}}\times\{y_{m+1}\}\times\cdots\times\{y_n\})\cap\OO$.
\end{proposition}
\begin{proof}
Let $H_m=H\cap\{1,\ldots,m\}$ for $m\ge1$ and $H_0=\emptyset$. 
We prove by induction on $m$ that $\SP^m_K(f)$ is induced by a stem function of the form
\begin{equation}\label{eq:gk}
G^m_K=
\sum_{H\in\Pm(n),H_m=\emptyset}e_HG_H.
\end{equation}
When $m=0$, $\SP^0_\emptyset(f)=f\in\SF(\OO)$. Let $m>0$ and assume that $\SP^{m-1}_{K\setminus\{m\}}(f)$ is a slice function in $\OO\setminus\R_{m-1}$, induced by a stem function 
\[
G^{m-1}_{K\setminus\{m\}}=
\sum_{H\in\Pm(n),H_{m-1}=\emptyset}e_HG_H.
\]
We can decompose the stem function as
\[
G^{m-1}_{K\setminus\{m\}}=
\sum_{H\in\Pm(n),H_m=\emptyset}e_HG_H+
\sum_{H\in\Pm(n),H_{m-1}=\emptyset,H\ni m}e_HG_H.
\]
Therefore, if $x=(\alpha_1+J_1\beta_1,\ldots,\alpha_n+J_n\beta_n)\in\OO\setminus\R_{m-1}$ and  $z=(\alpha_1+i\beta_1,\ldots,\alpha_n+i\beta_n)\in D$, it holds
\begin{align*}
\SP^{m-1}_{K\setminus\{m\}}(f)(x)&=
\sum_{H\in\Pm(n),H_m=\emptyset}J_HG_H(z)+
\sum_{H\in\Pm(n),H_{m-1}=\emptyset,H\ni m}J_HG_H(z)\\
&=
\sum_{H\in\Pm(n),H_m=\emptyset}J_HG_H(z)+J_m
\sum_{H\in\Pm(n),H_{m-1}=\emptyset,H\ni m}J_{H\setminus\{m\}}G_H(z)
\end{align*}
from which we infer that $\left(\SP^{m-1}_{K\setminus\{m\}}(f)\right)'_{s,x_m}$ and  $\left(\SP^{m-1}_{K\setminus\{m\}}(f)\right)^\circ_{s,x_m}$  are induced, respectively, by the stem functions of $n$ variables
\[
\sum_{H\in\Pm(n),H_{m-1}=\emptyset,H\ni m}e_{H\setminus\{m\}}\beta_m^{-1}G_H\text{\quad and\quad}
\sum_{H\in\Pm(n),H_m=\emptyset}e_HG_H.
\]
Since $(x_m g)'_{s,x_m}=\alpha_m (g)'_{s,x_m}+(g)^\circ_{s,x_m}$, it holds
\[
\left(x_m\SP^{m-1}_{K\setminus\{m\}}(f)\right)'_{s,x_m}=\I\left(
\sum_{H\in\Pm(n),H_{m-1}=\emptyset,H\ni m}e_{H\setminus\{m\}}\alpha_m\beta_m^{-1}G_H+
\sum_{H\in\Pm(n),H_m=\emptyset}e_HG_H\right)
\]
and formula \eqref{eq:gk} is proved for every $m$ and $K$. The form of the stem function \eqref{eq:gk} proves immediately the last statement, since
\[
\SP^m_K(f)(x)=
\sum_{H\in\Pm(n),H_m=\emptyset}J_HG_H(z)=\SP^m_K(f)(y)
\]
for every $x\in\s_{y_1}\times\cdots\times\s_{y_{m}}\times\{y_{m+1}\}\times\cdots\times\{y_n\}$.
\end{proof}

The family of functions $\{\SP^m_K(f)\}_{m,K}$ provides a new characterization of slice regularity in several variables, based on the one given in \cite[Theorem~3.23]{SeveralA}.  
Let $y=(y_1,\ldots,y_n)\in \OO$ and $h\in\{1,\ldots,n\}$. Denote by $\OO_{h}(y)$ the subset of $\hh$ defined by
\begin{equation}\label{def:omegay}
\OO_{h}(y):=\{a\in \hh:(y_1,\ldots,y_{h-1},a,y_{h+1},\ldots,y_n)\in\OO\}.
\end{equation}
As shown in \cite[\S2.3]{SeveralA}, if $\OO$ is axially symmetric in $\hh^n$, then every set $\OO_h(y)$ is axially symmetric in $\hh$.
Let $g:\OO\to \hh$ be a function. We say that \emph{$g$ is a slice function w.r.t.\ $x_h$} if, for each $y=(y_1,\ldots,y_n)\in\OO$, the restriction function $g_h^{(y)}:\OO_h(y)\to \hh$, sending $x_h$ into $g_h^{(y)}(x_h):=g(y_1,\ldots,y_{h-1},x_h,y_{h+1},\ldots,y_n)$, is a slice function. 
We say that \emph{$g$ is a slice-regular function w.r.t.\ $x_h$} if, for each $y\in\OO$, the function $g_h^{(y)}:\OO_{h}(y)\to \hh$ is a slice-regular function.

Let 
$m\in\{0,\ldots, n-1\}$ and $K\in\Pm(m)$. For $y\in\OO\setminus\R_\bullet$, let $\SP^{m,y}_K(f)$ denote the function defined for $x_{m+1}\in\OO_{m+1}(y)\setminus\R$  by
\begin{equation}\label{eq:smy}
\SP^{m,y}_K(f)(x_{m+1}):=\SP^m_K(f)^{(y)}(x_{m+1})=\SP^m_K(f)(y_1,\ldots,y_m,x_{m+1},y_{m+2},\ldots,y_n).
\end{equation}
In particular, for $m=0$ we have $\SP^{0,y}_\emptyset(x_1)=f(x_1,y_2,\ldots,y_n)$. 

The functions $\SP^m_K(f)$ provide another one-variable interpretation of slice regularity, equivalent to the one described in \cite[\S3.4]{SeveralA}). It holds:

\begin{proposition}\label{pro:1varChar}
A slice function 
$f\in\SF^1(\OO)$ in an axially symmetric open set $\OO\subseteq\hh^n$ is slice-regular in $\OO$ if and only if  $f$ is slice-regular w.r.t.\ $x_1$ and for every $m\in\{1,\ldots,n-1\}$ and $K\in\Pm(m)$, the functions $\SP^m_K(f)$ are slice-regular w.r.t.\ $x_{m+1}$, 
 i.e., $\SP^{m,y}_K(f)\in\sr(\OO_{m+1}(y)\setminus\R)$ for every $y\in\OO\setminus\R_\bullet$. 
\end{proposition}
\begin{proof}
Let $\OO=\OO_D$ and $f=\I(F)$, with $F=\sum_{K\in\Pm(n)}e_KF_K:D\to \hh\otimes\R^{2^n}$. As already observed above, the functions $\SP^m_K(f)$ can be written as combinations of the truncated spherical $\epsilon$-derivatives $\SD_\epsilon f$  
and vice versa. This fact implies that the assumption on the functions $\SP^m_K(f)$ is equivalent to the following condition: 

\smallskip
\noindent
\box1{\begin{minipage}[t][][c]{.05\textwidth}
$(\ast)$
\end{minipage}}
\hfill
\noindent
\box0{\begin{minipage}[t]{.90\textwidth}
\emph{The slice function $f$ is slice-regular w.r.t.\ $x_1$ and, for each $h\in\{2,\ldots,n\}$ and each function $\epsilon:\{1,\ldots,h-1\}\to\{0,1\}$, the truncated spherical $\epsilon$-derivative $\SD_\epsilon f$ of $f$ is slice-regular w.r.t.\ $x_h$, i.e., $(\SD_\epsilon f)^{(y)}\in\sr(\OO_h(y)\setminus\R)$ for every $y\in\OO\setminus\R_\bullet$. }\label{star}
\end{minipage}}
\medskip


\noindent
Theorem 3.23 in \cite{SeveralA} permits to conclude.
\end{proof}

We can refine the statement of Proposition \ref{pro:1varChar} and show that it continues to hold under the weaker assumption that $f\in\SF(\OO)\cap \mathcal C^0(\OO)$. 

\begin{proposition}\label{pro:1varChar2}
Let $f\in\SF(\OO)\cap \mathcal C^0(\OO)$ be a continuous slice function in an axially symmetric open set $\OO\subseteq\hh^n$.
Assume that $f$ is slice-regular w.r.t.\ $x_1$ and for every $m\in\{1,\ldots,n-1\}$ and $K\in\Pm(m)$, the functions $\SP^m_K(f)$ are slice-regular w.r.t.\ $x_{m+1}$, i.e., $\SP^{m,y}_K(f)\in\sr(\OO_{m+1}(y)\setminus\R)$ for every $y\in\OO\setminus\R_\bullet$. Then $f$ is slice-regular in $\OO$.
\end{proposition}

\begin{proof}
Let $\OO=\OO_D$ and $f=\I(F)$, with $F=\sum_{K\in\Pm(n)}e_KF_K:D\to \hh\otimes\R^{2^n}$. As in the proof of Proposition \ref{pro:1varChar}, 
the assumption on the functions $\SP^m_K(f)$ is equivalent to the above condition (*). 
In view of Proposition \ref{pro:1varChar}, $f$ is slice-regular in $\OO\setminus\R_\bullet$. 
Therefore for every $h\in\{1,\ldots,n\}$ and every $K\in\Pm(n)$ with $h\not\in K$, it holds
\begin{equation}\label{eq:CR}
\partial_{\alpha_h}F_{K}(z)=\partial_{\beta_h}F_{K\cup\{h\}}(z),\quad \partial_{\beta_h}F_{K}(z)=-\partial_{\alpha_h}F_{K\cup\{h\}}(z)
\end{equation}
for all $z\in D_\bullet$. Let $J\in\s$ be fixed, let $\OO_D(J)=\OO\cap(\C_J)^n$ and let $f_J:\OO_D(J)\to \hh$ be the restriction of $f$ on $\OO_D(J)$. For every $h\in\{1,\ldots,n\}$ and $x=(\alpha_1+J\beta_1,\ldots,\alpha_n+J\beta_n)\in\OO_D(J)$, it holds
\[
f_J(x)=
\sum_{K\in\Pm(n),h\not\in K}J^{|K|}(F_K(z)+JF_{K\cup\{h\}}(z)).
\] 
Thanks to \eqref{eq:CR}, we get that
\begin{equation}\label{eq:CR_fJ}
\partial_{\alpha_h}f_J(x)+J\partial_{\beta_h}f_J(x)=0
\end{equation}
for all $x\in\OO_D(J)\setminus\R^n$ and all $h\in\{1,\ldots,n\}$, i.e., $f_J$ is separately holomorphic in $x_1,\ldots,x_n$ in $\OO_D(J)\setminus\R^n$ w.r.t.\ the complex structure defined by the left multiplication by $J$. Since $f_J$ is continuous on $\OO_D(J)$, by Painlev\'e's Theorem applied to the $\cc_J$-components of $f_J$ we get that $f_J$ is separately holomorphic in the whole set $\OO_D(J)$. Osgood Lemma permits to conclude that $f_J$ is holomorphic in $\OO_D(J)$. In particular, in view of formula (8) in \cite{SeveralA}, 
the components $F_K$ of $F$ are of class $C^1$ in $D$, i.e., $f\in\SF^1(\OO)$. 
Finally, \cite[Proposition~3.13]{SeveralA} yields that $f$ is slice-regular in the whole set $\OO$. 
\end{proof}

\begin{remark}\label{rem:extend}
If $f$ is a slice function of class $\mathcal C^1$ with respect to one variable in $\OO'\subseteq\hh$, then $f$ is slice-regular if and only if $\sd f=-\dcf f$ in $\OO'\setminus\rr$ \cite[Corollary~3.6.2]{Harmonicity} (observe that in Ref.\cite{Harmonicity} the operator $\dcf$ differs by a factor 2 w.r.t.\ our definition).
Hence the spherical derivative $\sd f$ extends real analytically to $\OO'$. Definition \ref{def:sp} and Proposition~\ref{pro:1varChar} imply that if $f\in\sr(\OO)$, then the functions $\SP^m_K(f)$ extend real analytically to the whole $\OO$. We will denote by the same symbol $\SP^m_K(f)$ also these extensions on $\OO$.
\end{remark}

We now recall the announced extension of the Almansi-type decomposition to slice-regular functions of several variables, proved in \cite[Theorem 5.1]{Binosi1}. 

\begin{theorem}[Almansi Theorem in several quaternionic variables \cite{Binosi1}]
\label{teo:Almansi_several}
Let $f$ be slice-regular in an axially symmetric open set $\OO\subseteq\hh^n$. For every fixed $m\in\{1,\ldots, n\}$, the family of $2^m$ quaternionic-valued real analytic slice functions $\{\SP^m_K(f)\}_{K\in\Pm(m)}$ decomposes $f$ in the form 
\[
f(x)=\sum_{K\in\Pm(m)}(-\overline{x})_{K^c}\SP^m_K(f)(x)=\sum_{K\in\Pm(m)}(-\overline{x})_{K^c}\cdot\SP^m_K(f)(x)
\]
for every $x\in\OO$, where $K^c=\{1,\ldots,m\}\setminus K$,  
$(-\overline{x})_K=(-\overline{x}_{k_1})(-\overline{x}_{k_2})\cdots (-\overline{x}_{k_p})$ for every $K=\{k_1,\ldots,k_p\}\in\Pm(m)$ with $k_1<k_2<\cdots <k_p$ and $\cdot$ denotes the slice product of slice functions.
Moreover, if $m=n$, then every function $\SP^n_K(f)$ is separately zonal harmonic (with pole 1) in the variables $x_1,\ldots, x_n$. If $m<n$, then the functions $\SP^m_K(f)$ are separately zonal harmonic (with pole 1) in  $x_1,\ldots, x_m$ and slice-regular w.r.t.\ $x_{m+1}$ in $\OO_{m+1}(y)$.
\end{theorem}

\subsection{Decomposition for slice functions of several quaternionic variables}\label{sec:AlmansiDecomposition2}

In the definition of Wirtinger operators (see \S\ref{sec:Wirtinger operators}), we will need to extend the Almansi Theorem in several quaternionic variables (Theorem \ref{teo:Almansi_several}) to slice, not necessarily regular, functions on open subsets of $\hh^n$. As one would expect, the zonal decomposition obtained in this case has no more harmonicity and regularity properties. 
We first need a one-variable characterization for general slice, not necessarily regular, functions.

\begin{proposition}\label{pro:1varChar_slice}
If $f:\OO\to\hh$  is a slice function in an axially symmetric open set $\OO\subseteq\hh^n$, then $f$ is a slice function w.r.t.\ $x_1$ and for every $m\in\{1,\ldots,n-1\}$ and $K\in\Pm(m)$, the functions $\SP^m_K(f)$ are slice functions w.r.t.\ $x_{m+1}$, 
 i.e., $\SP^{m,y}_K(f)\in\SF(\OO_{m+1}(y)\setminus\R)$ for every $y\in\OO\setminus\R_\bullet$. 
\end{proposition}
\begin{proof}
As in the proof of the previous proposition, the assumption on the functions $\SP^m_K(f)$ is equivalent to the following condition: 

\smallskip
\noindent
\box1{\begin{minipage}[t][][c]{.05\textwidth}
$(\ast\ast)$
\end{minipage}}
\hfill
\noindent
\box0{\begin{minipage}[t]{.90\textwidth}
\emph{$f$ is a slice function w.r.t.\ $x_1$ and, for each $h\in\{2,\ldots,n\}$ and each function $\epsilon:\{1,\ldots,h-1\}\to\{0,1\}$, the truncated spherical $\epsilon$-derivative $\SD_\epsilon f$ of $f$ is a slice function w.r.t.\ $x_h$, i.e., $(\SD_\epsilon f)^{(y)}\in\SF(\OO_h(y)\setminus\R)$ for every $y\in\OO\setminus\R_\bullet$. }\label{2star}
\end{minipage}}
\medskip

\noindent
If $f\in\SF(\OO)$, Proposition~2.23 in \cite{SeveralA} implies that $f$ satisfies (**). 
\end{proof}

\begin{theorem}
\label{teo:slice_Almansi_several}
Let $f\in\SF(\OO)$ be a slice function in an axially symmetric open set $\OO\subseteq\hh^n$.  Then for every fixed $m\in\{1,\ldots, n\}$ it holds 
\begin{equation}\label{eq:slice_Almansi_several}
f(x)=\sum_{K\in\Pm(m)}(-\overline{x})_{K^c}\SP^m_K(f)(x)=\sum_{K\in\Pm(m)}(-\overline{x})_{K^c}\cdot\SP^m_K(f)(x)
\end{equation}
for every $x\in\OO\setminus\R_m$. 
\end{theorem}
\begin{proof}
We prove \eqref{eq:slice_Almansi_several} by induction on $m$. 
Proposition \ref{pro:1varChar_slice} says that the slice function $f$ is a slice function w.r.t.\ $x_1$ and the functions $\SP^{m}_{K}(f)\in\SF(\OO\setminus\R_m)$ are slice functions w.r.t.\ $x_{m+1}$ for every $m<n$. 
When $m=1$, then 
\[
f(x)=(x_1f)'_{s,x_1}-\overline x_1\,(f)'_{s,x_1}=\SP^1_{\{1\}}(f)(x)-\overline x_1\,\SP^1_\emptyset(f)(x)
\]
for every $x\in\OO\setminus\R_1$. Assume that formula \eqref{eq:slice_Almansi_several} is valid for the integer $m<n$. If $K\in\Pm(m)$, applying again the same decomposition valid for slice functions w.r.t.\ $x_{m+1}$, we get
\begin{align}\label{eq:smK}
\SP^m_K(f)&=(x_{m+1}\SP^m_K(f))'_{s,x_{m+1}}-\overline x_{m+1}\,(\SP^m_K(f))'_{s,x_{m+1}}=\\\notag
&=\SP^{m+1}_{K\cup\{m+1\}}(f)-\overline x_{m+1}\,\SP^{m+1}_K(f)
\end{align}
on $\OO\setminus\R_{m+1}$.
By the induction hypothesis and \eqref{eq:smK},
\begin{align*}
f(x)&=
\sum_{K\in\Pm(m)}(-\overline{x})_{K^c}\SP^m_K(f)(x)=\\
&=
\sum_{K\in\Pm(m)}(-\overline{x})_{K^c}\SP^{m+1}_{K\cup\{m+1\}}(f)(x)-\sum_{K\in\Pm(m)}(-\overline{x})_{K^c}\overline x_{m+1}\,\SP^{m+1}_K(f)(x)=\\
&=
\sum_{H\in\Pm(m+1)}(-\overline{x})_{H^c}\SP^{m+1}_H(f)(x)
\end{align*}
for every $x\in\OO\setminus\R_{m+1}$. We now prove the second equality in \eqref{eq:slice_Almansi_several}. Let $K^c=(x_{h_1},\ldots, x_{h_p})$, with $1\le h_1<\cdots<h_p\le m$. We apply iteratively \cite[Proposition~2.52]{SeveralA} to the pointwise product $x_{h_1}\cdots x_{h_p}\SP^m_K(f)(x)$. Since $\SP^m_K(f)$ is $\{m+1,\ldots,n\}$-reduced (see formula \eqref{eq:gk} expressing the stem function inducing $\SP^m_K(f)$ and see \cite[Definition~2.50]{SeveralA} for the concept of $H$-reduced slice function) and every $\overline x_{h_l}$ is slice-preserving and $\{h_l\}$-reduced, we get that the pointwise product $\overline x_{K^c}=\overline x_{h_1}\cdots \overline x_{h_p}$ coincides with the slice product $\overline x_{h_1}\cdot\ldots\cdot \overline x_{h_p}$, and
\[
\overline x_{h_1}\cdots \overline x_{h_p}\SP^m_K(f)=\overline x_{h_1}\cdots \overline x_{h_{p-1}}\left(\overline x_{h_p}\cdot\SP^m_K(f)\right)=\cdots=\overline x_{h_1}\cdot\ldots\cdot \overline x_{h_p}\cdot\SP^m_K(f)=\overline x_{K^c}\cdot \SP^m_K(f).
\]
\end{proof}

The decomposition of Theorem \ref{teo:slice_Almansi_several} is uniquely determined by the slice function $f$, as shown in the next proposition.

\begin{proposition}
\label{pro:converse_slice_Almansi_several}
Let $f\in\SF(\OO)$ be a slice function in an axially symmetric open set $\OO\subseteq\hh^n$. Let $m\in\{1,\ldots, n\}$ be fixed. Assume that there exists a collection of functions $\{g_K\}_{K\in\Pm(m)}$ 
such that for every $y\in\OO\setminus\R_m$, $g_K$ is constant on the set $\s_{y_1}\times\cdots\times\s_{y_{m}}\times\{y_{m+1}\}\times\cdots\times\{y_n\}$ and it holds 
\begin{equation}\label{eq:converse-Almansi_several}
f(x)=\sum_{K\in\Pm(m)}(-\overline{x})_{K^c}g_K(x)
\end{equation}
for every $x\in\OO\setminus\R_m$.  Then $g_K=\SP^m_K(f)$ for every $K\in\Pm(m)$. 
\end{proposition}
\begin{proof}
As a consequence of Theorem \ref{teo:slice_Almansi_several}, it suffices to prove the thesis with $f\equiv0$.
We proceed by induction on $m$. If $m=1$, formula \eqref{eq:converse-Almansi_several} reduces to $g_{\{1\}}-\overline x_1 g_\emptyset=0$. From Remark \ref{rem:sd} we deduce 
\be\label{eq:h1h2}
g_{\{1\}}=(x_1f)'_{s,x_1}=\SP^1_{\{1\}}(f)\equiv0,\quad  g_\emptyset=(f)'_{s,x_1}=\SP^1_{\emptyset}(f)\equiv0.
\eeq
Let $m>1$ and assume the statement valid for the integer $m-1$. Rewrite \eqref{eq:converse-Almansi_several} as
\begin{align*}
0&=
\sum_{K\in\Pm(m)}(-\overline x)_{K^c}g_K=\\
&=
\sum_{K\in\Pm(m),K\not\ni m}(-\overline x)_{K^c} g_K+\sum_{K\in\Pm(m),K\not\ni m}(-\overline x)_{(K\cup\{m\})^c} g_{K\cup\{m\}}=\\
&=
\sum_{H\in\Pm(m-1)}(-\overline x)_{H^c}\left(-\overline x_m g_H+ g_{H\cup\{m\}}\right).
\end{align*}
The functions $-\overline x_m g_H+ g_{H\cup\{m\}}$ are constant on $\s_{y_1}\times\cdots\times\s_{y_{m-1}}\times\{y_{m}\}\times\cdots\times\{y_n\}$ for every $y\in\OO\setminus\R_m$. From the inductive hypothesis, we get 
\[
-\overline x_m g_H+ g_{H\cup\{m\}}=0\text{\quad for every }H\in\Pm(m-1). 
\]
Since $g_H$ and $g_{H\cup\{m\}}$ are constant on $\s_{y_1}\times\cdots\times\s_{y_{m}}\times\{y_{m+1}\}\times\cdots\times\{y_n\}$, in view of Remark \ref{rem:sd} we get $g_H=g_{H\cup\{m\}}\equiv0$ for every $H\in\Pm(m-1)$, i.e., $g_K\equiv0$ for every $K\in\Pm(m)$.
\end{proof}

The functions $\SP^m_K(f)$ defined in \eqref{def:sp} were introduced in \cite{Binosi1} when $f$ is a slice-regular function in an axially symmetric domain $\OO$. 
Now we extend the definition to any open set $  U \subseteq\hh^n$ and any $f$ of class $\mathcal C^n(  U,\hh )$ by means of the Cauchy-Riemann-Fueter operators $\dcf_{x_m}$.

Let $n>1$,  $m\in\{1,\ldots,n\}$ and let $\dcf_{x_m}$ denote the \emph{Cauchy-Riemann-Fueter operator w.r.t.\ the quaternionic variable $x_m=x_{m_0}+ix_{m_1}+jx_{m_2}+kx_{m_3}$}, namely
\[\dcf_{x_m} =\frac12\left(\dd{}{x_{m_0}}+i\dd{}{x_{m_1}}+j\dd{}{x_{m_2}}+k\dd{}{x_{m_3}}\right).\]

Let $f\in\mathcal C^n(  U ,\hh)$, with $  U \subseteq\hh^n$ open (not necessarily axially symmetric). Define recursively 
functions $\A^m_K(f)$, for every $m\in\{0,\ldots, n\}$ and every $K\in\Pm(m)$, as follows: 
\begin{equation}\label{eq:amk}
\begin{cases}
\A^0_\emptyset(f):=f&\text{\quad for }m=0,\\
\A^m_K(f):=-\dcf_{x_m}\left(x_m^{\mathbf{1}_K(m)}\A^{m-1}_{K\setminus\{m\}}(f)\right) &\text{\quad for }m\in\{1,\ldots,n\}.
\end{cases}
\end{equation}
The functions $\A^m_K(f)$ are of class $\mathcal C^{n-m}(  U,\hh)$.  
From Proposition~\ref{pro:1varChar}, Corollary 3.6.2 in \cite{Harmonicity} and Remark \ref{rem:extend}, it follows that if $f\in\mathcal C^n(\OO,\hh)$ is a slice function in an axially symmetric set $\OO$, then $f\in\sr(\OO)$ if and only if $\SP^m_K(f)=\A^m_K(f)$ for every $m$ and $K$. In this case we will call the functions $\A^m_K(f)$ the \emph{Almansi components} of the slice-regular function $f$. 

\section{Locally slice-regular functions of several variables}\label{sec:Locally_slice-regular}

We now recall and generalize some definitions from \cite{CR_S_operators}. They deal with the possibility of extending the concept of slice-regularity to domains in $\hh$ or in $\hh^n$ that are not necessarily axially symmetric. We begin with the one-dimensional case, already introduced in \cite[Definition~13]{CR_S_operators}.

\begin{definition}\label{def:loc}
Let $  U \subseteq \hh$ be any open set and let $f\in\mathcal C^1(  U,\hh )$. The function $f$ is called \emph{locally strongly slice-regular} 
 if the following properties hold:
\begin{enumerate}
\item[(i)] For every $I\in S$ such that $  U _I:=  U \cap \cc_I\not=\emptyset$, the restriction 
\[f_I:=f_{|  U _I}:(  U _I,L_I)\to(\hh,L_I)\]
is holomorphic w.r.t.\ the complex structure $L_I$ defined by left multiplication by $I$.
\item[(ii)] The functions $\dcf f$ and $\dcf(xf)$ are locally constant on $\s_y\cap   U $ for every $y\in   U $.
\end{enumerate}
If property (i) holds and the functions $\dcf f$ and $\dcf(xf)$ in property (ii) are constant on $\s_y\cap   U $ for every $y\in   U $, then $f$ is called \emph{strongly slice-regular}. 
\end{definition}

Observe that condition (i) is exactly the assumption of \emph{$C$-regularity} (or \emph{Cullen regularity}) introduced by Gentili and Struppa in the original source of the theory \cite[Definition 1.2]{GeSt2007Adv}. Condition (ii) has the role to connect the restrictions of the function on different complex slices of $\hh$. It reflects what happens for quaternionic polynomials 
$\sum_j x^ja_j$.  

From Definition \ref{def:loc} it follows immediately that strongly slice-regular functions are locally strongly slice-regular. 
As observed in \cite[Proposition~17]{CR_S_operators}, if $\OO\subseteq\hh$ is an axially symmetric domain, then $f\in\sr(\OO)$ if and only if $f$ is strongly slice-regular in $\OO$, or, equivalently, $f$ is locally strongly slice-regular in $\OO$. 
As we will see below, in general this property does not hold on non axially symmetric sets.

\begin{definition}
If $U\subseteq\hh$ is any open set, we denote by $\sr(U)$ the set of strongly slice-regular functions in $U$, and by  $\sr_{loc}(U)$ the set of locally strongly slice-regular functions in $U$. For simplicity, the elements of $\sr(U)$ will be called \emph{slice-regular functions} on $U$. Thanks to \cite[Proposition~17]{CR_S_operators}, this notation is consistent with the usual one if $U$ is axially symmetric. Similarly, the elements of $\sr_{loc}(U)$ will be called 
\emph{locally slice-regular functions} on $U$. 
\end{definition}

The one-variable characterization of slice-regularity obtained in \cite[Theorem~3.23]{SeveralA} and in Proposition \ref{pro:1varChar} suggests a possible extension of Definition \ref{def:loc} to several variables. 

For $y\in U $, let $U _{m+1}(y)$ be the set defined as in \eqref{def:omegay}, that is
\[
U_{m+1}(y)=\{a\in \hh:(y_1,\ldots,y_{m},a,y_{m+2},\ldots,y_n)\in U\},
\]
and let $\A^{m,y}_K(f)$ be the one-variable functions of $x_{m+1}\in U _{m+1}(y)$ defined as in formula \eqref{eq:smy}:  
\begin{equation}\label{eq:amy}
\A^{m,y}_K(f)(x_{m+1}):=\A^{m}_K(f)^{(y)}(x_{m+1})=\A^m_K(f)(y_1,\ldots,y_m,x_{m+1},y_{m+2},\ldots,y_n).
\end{equation}
In particular, for $m=0$ we have $\A^{0,y}_\emptyset(x_1)=f(x_1,y_2,\ldots,y_n)$.

\begin{definition}\label{def:loc_several}
Let $ U \subseteq \hh^n$ ($n>1$) be any open set and let $f\in\mathcal C^n( U,\hh)$. 
The function $f$ is called \emph{locally strongly slice-regular}, or simply \emph{locally slice-regular}, and we write $f\in\sr_{loc}( U )$, if for every $y\in U $, every $m\in\{0,\ldots, n-1\}$ and every $K\in\Pm(m)$, the function $\A^{m,y}_K(f)$ belongs to $\sr_{loc}( U _{m+1}(y))$. 

If it holds that for every $y\in U $, every $m\in\{0,\ldots, n-1\}$ and every $K\in\Pm(m)$, the function $\A^{m,y}_K(f)$ belongs to $\sr( U _{m+1}(y))$, then  $f$ is called \emph{strongly slice-regular}. 
By definition,  every strongly slice-regular function in $U$ belongs to $\sr_{loc}( U )$. 
\end{definition}

For example, when $n=2$, Definition \ref{def:loc_several} means that $f\in\sr_{loc}( U )$ if and only if
for every $y\in U $, $x_1\mapsto f(x_1,y_2)\in\sr_{loc}( U _1(y))$, i.e., $f$ is locally 
slice-regular in the first variable $x_1$, and the functions $x_2\mapsto(\dcf_{x_1}f)(y_1,x_2)$ and $x_2\mapsto(\dcf_{x_1}(x_1f))(y_1,x_2)$ belong to $\sr_{loc}( U _2(y))$, i.e., they are locally 
slice-regular in the second variable $x_2$.

\begin{proposition}\label{pro:loc_several}
Let $ U \subseteq \hh^n$ ($n>1$) be any open set and let $f\in\mathcal C^n( U,\hh )$. The function $f$ is locally 
slice-regular in $U$ if and only if for every $y\in U $, 
it holds:
\begin{enumerate}
\item[(i)] 
For every $m\in\{0,\ldots, n-1\}$, $K\in\Pm(m)$ and $I\in \s$ such that $ U _{m+1}^I(y):= U _{m+1}(y)\cap \cc_I\not=\emptyset$, the restriction 
\[{\A^{m,y}_K(f)}_I:={\A^{m,y}_K(f)}_{| U _{m+1}^I(y)}:( U _{m+1}^I(y),L_I)\to(\hh,L_I)\]
 is holomorphic w.r.t.\ the complex structure $L_I$ defined by left multiplication by $I$.
\item[(ii)] For every $m\in\{1,\ldots, n\}$ and $K\in\Pm(m)$, the functions $\A^{m}_{K}(f)$ are locally constant on the set $(\s_{y_1}\times\cdots\times\s_{y_{m}}\times\{y_{m+1}\}\times\cdots\times\{y_n\})\cap U $.
\end{enumerate}
The function $f$ is strongly slice-regular if and only if property (i) holds and the functions $\A^{m}_K(f)$ in (ii) are constant on the set $(\s_{y_1}\times\cdots\times\s_{y_{m}}\times\{y_{m+1}\}\times\cdots\times\{y_n\})\cap U $.
\end{proposition}
\begin{proof}
The statements are simply a rephrasing of Definition \ref{def:loc_several} in view of Definition \ref{def:loc}. 
\end{proof}

\begin{proposition}\label{pro:sr_is_locally_sr}
Let  $\OO\subseteq\hh^n$ be an axially symmetric domain and $f\in\mathcal C^n(\OO,\hh)$. Then $f$ is locally slice-regular in $\OO$ if and only if $f$ is strongly slice-regular in $\OO$.
\end{proposition}
\begin{proof}
Let $m\in\{0,\ldots,n-1\}$. As already observed, since $\OO$ is axially symmetric, the sets $\OO_{m+1}(y)$ are axially symmetric in $\hh$  for every $y\in\OO$. The statement follows from \cite[Proposition~17]{CR_S_operators}.
\end{proof}

Given any $x=(x_1,\ldots,x_n)\in\hh^n$, let $\s_x=\s_{x_1}\times\cdots\times\s_{x_n}$, where $\s_{x_h}=\alpha_h+\s\beta_h$ if $x_h=\alpha_h+J_h\beta_h$ for some $\alpha_h,\beta_h\in\R$ and $J_h\in\s$. Note that $\s_x$ is the smallest axially symmetric subset of $\hh^n$ containing $x$.

The condition of strong slice regularity guarantees that the function can be extended as a slice-regular function to an axially symmetric set. The \emph{symmetric completion} of a subset $ U $ of $\hh^n$ is the set 
\[\widetilde  U =\bigcup_{x\in  U }\s_x.\]
It is the smallest axially symmetric subset of $\hh^n$ containing $ U $.

\begin{theorem}[Global extendibility of strongly slice-regular functions]\label{teo:extension}
Let $ U \subseteq\hh^n$ be open and let $f\in\mathcal C^n( U,\hh )$. The function $f$ is strongly slice-regular in $ U $ if and only if it can be extended to a slice-regular function on the symmetric completion $\widetilde  U $. If this is the case, the extension is unique.
\end{theorem}
\begin{proof}
The one-dimensional case has been proved in the more general setting of hypercomplex subspaces in \cite[Theorem~15]{CR_S_operators}. In the notations of that result one can take $\hh$ as the hypercomplex subspace $M$, $m=3$, $c_m=-1$. Assume $n>1$ and let $f$ be strongly slice-regular in $U$. 
In particular, for every $y\in U $ and $I\in\s$ the restrictions of $\A^{0,y}_\emptyset(f)(x_1)=f(x_1,y_2,\ldots,y_n)$ and $x_1\A^{0,y}_\emptyset(f)$ for $x_1\in U _1^I(y)= U _1(y)\cap\cc_I$ are holomorphic. This means that $\difbar_{x_1}f=\difbar_{x_1}(x_1f)=0$ on $ U _1(y)$ for every $y\in U $, i.e., $\difbar_{x_1}f=\difbar_{x_1}(x_1f)=0$ on $ U $. Repeating the argument of the proof of \cite[Theorem~15]{CR_S_operators} and using the equalities $\A^{1}_{\emptyset}(f)=-\dcf_{x_1}f$ and $\A^{1}_{\{1\}}(f)=-\dcf_{x_1}(x_1f)$, we get the zonal decomposition of $f$ w.r.t.\ $x_1$, that is
\[
f=\A^{1}_{\{1\}}(f)-\overline x_1\A^{1}_\emptyset(f)
\]
on $ U $, with $\A^{1}_\emptyset(f)$ and $\A^{1}_{\{1\}}(f)$ of class $\mathcal C^{n-1}$ in $ U $. Moreover, as in \cite[Theorem~15]{CR_S_operators} we get that $\A^{1,y}_\emptyset(f)$ and $\A^{1,y}_{\{1\}}(f)$ are real analytic w.r.t.\ $x_1\in U _1(y)\setminus\R$. Therefore $f$ is real analytic w.r.t.\ $x_1$ in $ U \setminus\R_\bullet$. 

Since the functions $\A^{1,y}_K(f)$ belong to $\sr( U _2(y))$, we can repeat the same argument and obtain on $ U $
\[
f=\A^{1}_{\{1\}}(f)-\overline x_1\A^{1}_\emptyset(f)=
\left(\A^{2}_{\{1,2\}}(f)-\overline x_2\A^{2}_{\{1\}}(f)\right)-\overline x_1\left(\A^{2}_{\{2\}}(f)-\overline x_2\A^{2}_\emptyset(f)\right).
\]
The functions $\A^{2}_K(f)$ are of class $\mathcal C^{n-2}$ in $ U $, with restrictions  $\A^{2,y}_K(f)$ real analytic w.r.t.\ $x_2\in U _2(y)\setminus\R$ for all $y\in U $. Therefore the functions $A^1_K(f)=A^2_{K\cup\{2\}}(f)-\overline x_2 A^2_K(f)$ are real analytic w.r.t.\ $x_1$ and $x_2$ in $U\setminus\R_\bullet$ and $f$ is real analytic w.r.t.\ $x_1$ and $x_2$ in $ U \setminus\R_\bullet$. 
Proceeding inductively, we obtain, for every $m\in\{1,\ldots,n\}$, the decomposition
\begin{equation}\label{eq:fm}
f(x)=\sum_{K\in\Pm(m)}(-\overline x)_{K^c}\A^{m}_K(f)(x)\text{\quad for every }x\in U,
\end{equation}
with $A^m_K(f)\in\mathcal C^{n-m}(U,\hh)$ and restriction functions  $\A^{m,y}_K(f)$ that are real analytic w.r.t.\ $x_m\in U _m(y)\setminus\R$ for all $y\in U $. Moreover, for $m\ge1$ and $K'\in\Pm(m-1)$ the functions $A^{m-1}_{K'}(f)$ satisfy
\begin{equation}\label{eq:am1}
A^{m-1}_{K'}(f)=A^{m}_{K'\cup\{m\}}(f)-\overline x_m A^m_{K'}(f)
\end{equation}
in $U$ and are separately real analytic w.r.t.\ $x_{m-1}$ and $x_m$ in $U\setminus\R_\bullet$.  
As a final step, we arrive at the full decomposition
\[
f(x)=\sum_{K\in\Pm(n)}(-\overline x)_{K^c}\A^n_K(f)(x)\text{\quad for every }x\in U ,
\]
with $\A^{n}_K(f)\in\mathcal C^{0}( U,\hh )$. 
From \eqref{eq:fm} we infer that $f$ is separately real analytic w.r.t.\ all variables $x_1,\ldots,x_n$ in $ U \setminus\R_\bullet$. 

By condition (ii) of Proposition~\ref{pro:loc_several}, every function $\A^{m}_K(f)$, being constant on the set $(\s_{y_1}\times\cdots\times\s_{y_{m}}\times\{y_{m+1}\}\times\cdots\times\{y_n\})\cap U $, admits a continuous extension  $\widetilde\A^{m}_K(f)$ to the open set
\[
\widetilde U _{m}:=\bigcup_{y\in U }\s_{y_1}\times\cdots\times\s_{y_{m}}\times\{y_{m+1}\}\times\cdots\times\{y_n\}.
\]
It holds $ U \subseteq\widetilde U _1\subseteq\cdots\subseteq\widetilde U _n=\widetilde U $. Observe that these extensions are compatible with equation \eqref{eq:am1}, in the sense that it holds
\begin{equation}\label{eq:am1tilde}
\widetilde A^{m-1}_{K'}(f)=\widetilde A^{m}_{K'\cup\{m\}}(f)-\overline x_m \widetilde A^m_{K'}(f)
\end{equation}
on $\widetilde U_{m-1}$ for every $m\ge1$ and $K'\in\Pm(m-1)$. 

As in the one-variable case \cite[Theorem~15]{CR_S_operators}, for every $K\in\Pm(m)$  and every $y\in U $, the functions $\widetilde\A^{m,y}_K(f)$ are 
real analytic w.r.t.\ $x_{m}$ in $\widetilde U_m\setminus\R_\bullet$. 
In particular, when $m=n$, the functions $\{\A^n_K(f)\}_{K\in\Pm(n)}$ 
extend to functions $\widetilde \A^n_K(f)\in\mathcal C^0(\widetilde U,\hh )$, constant on every set $\s_y\subseteq\widetilde U $. 
The functions $\widetilde \A^n_K(f)$ are then circular slice functions in $\widetilde U $. 

Let $K^c=\{h_1,\ldots,h_p\}$, with $h_1<\cdots <h_p$. The pointwise product $(-\overline x)_{K^c}$ coincides with the slice product $(-\overline{x}_{h_1})\cdot(-\overline{x}_{h_2})\cdots (-\overline{x}_{h_p})$ (see \cite[Proposition~2.52]{SeveralA}). Moreover, since $\widetilde \A^n_K(f)$ is circular, the pointwise product $(-\overline x)_{K^c}\widetilde\A^n_K(f)(x)$ coincides with the slice product (see \cite[Lemma~2.51]{SeveralA})
\[
(-\overline{x}_{h_1})\cdot(-\overline{x}_{h_2})\cdot\ldots\cdot (-\overline{x}_{h_p})\cdot\widetilde\A^n_K(f)(x).
\]
Therefore the function defined on $\widetilde U $ by
\begin{equation}\label{eq:ftilde}
\widetilde f(x):=\sum_{K\in\Pm(n)}(-\overline x)_{K^c}\widetilde\A^n_K(f)(x)\text{\quad for }x\in\widetilde U 
\end{equation}
is a continuous slice function in $\widetilde U $, separately real analytic w.r.t.\ $x_1,\ldots,x_n$ in $\widetilde{ U }\setminus\R_\bullet$, and such that $\widetilde f=f$ on $ U $. 

By backward induction on $m\ge2$, we show that every function $\widetilde \A^m_K(f)$ extends as a continuous slice function in $\widetilde U$, real analytic w.r.t.\ $x_m$ and $x_{m+1}$ (if $m<n$) in $\widetilde U\setminus\R_\bullet$, such that 
\begin{equation}\label{eq:ftildem}
\widetilde f(x)=\sum_{K\in\Pm(m)}(-\overline x)_{K^c}\widetilde\A^m_K(f)(x)\text{\quad for every }x\in\widetilde U.
\end{equation}
The case $m=n$ is equation \eqref{eq:ftilde}. 
Thanks to \eqref{eq:am1tilde}, it holds  $\widetilde \A^{n-1}_K(f)=\widetilde \A^{n}_{K\cup\{n\}}(f)-\overline x_n\widetilde \A^{n}_K(f)$ on $\widetilde U_{n-1}$. Then we can set $\widetilde \A^{n-1}_K(f):=\widetilde \A^{n}_{K\cup\{n\}}(f)-\overline x_n\widetilde \A^{n}_K(f)$ on $\widetilde U$. Using \eqref{eq:am1tilde} for $m=n-1,\ldots,2$, and setting  
\begin{equation}\label{eq:amtilde}
\widetilde \A^{m-1}_K(f):=\widetilde \A^{m}_{K\cup\{m\}}(f)-\overline x_m\widetilde \A^{m}_K(f)
\end{equation}
 on $\widetilde U$, we get the required extensions, which we denote with the same symbols $\widetilde \A^m_K(f)$. 
 Note that for every $y\in\widetilde U$, the extended functions $\widetilde \A^m_K(f)$ are constant on the set $\s_{y_1}\times\cdots\times\s_{y_{m}}\times\{y_{m+1}\}\times\cdots\times\{y_n\} $. 
Assume that \eqref{eq:ftildem} holds for an integer $m$. We prove it for the value $m-1$. From \eqref{eq:amtilde}, we get
\begin{align*}
\widetilde f&=
\sum_{K\in\Pm(m)}(-\overline x)_{K^c}\widetilde\A^m_K(f)=\\
&=
\sum_{K\in\Pm(m),K\not\ni m}(-\overline x)_{K^c}\widetilde\A^m_K(f)+\sum_{K\in\Pm(m),K\not\ni m}(-\overline x)_{(K\cup\{m\})^c}\widetilde\A^m_{K\cup\{m\}}(f)=\\
&=
\sum_{H\in\Pm(m-1)}(-\overline x)_{H^c}\left(-\overline x_m\widetilde\A^m_H(f)+\widetilde\A^m_{H\cup\{m\}}(f)\right)=
\sum_{H\in\Pm(m-1)}(-\overline x)_{H^c}\widetilde \A^{m-1}_K(f).
\end{align*}
Now we prove that $\widetilde\A^{m}_K(f)=\SP^{m}_K(\widetilde f)$ in $\widetilde U\setminus\R_\bullet$ for every $m\in\{1,\ldots,n\}$ and every $K\in\Pm(m)$. 
From \eqref{eq:ftildem} with $m=1$, we have $\widetilde f=\widetilde\A^{1}_{\{1\}}(f)-\overline x_1 \widetilde\A^{1}_\emptyset(f)$ on $\widetilde U$. From this equation and Remark \ref{rem:sd} it follows that 
\[
\widetilde\A^{1}_\emptyset(f)=(\widetilde f)'_{s,x_1}=\widetilde\SP^{1}_\emptyset(\widetilde f)\text{\quad and\quad}
\widetilde\A^{1}_{\{1\}}(f)=(x_1\widetilde f)'_{s,x_1}=\widetilde\SP^{1}_{\{1\}}(\widetilde f).
\]
Let $m>1$ and assume that $\widetilde\A^{m-1}_K(f)=\SP^{m-1}_K(\widetilde f)$ for every $K\in\Pm(m-1)$. From \eqref{eq:amtilde} and Remark \ref{rem:sd} we get 
\begin{align*}
\widetilde \A^{m}_K(f)&=(\widetilde\A^{m-1}_K(f))'_{s,x_m}=(\SP^{m-1}_K(\widetilde f))'_{s,x_m}=\SP^m_K(\widetilde f),\\
\widetilde \A^{m}_{K\cup\{m\}}(f)&=(x_m\widetilde\A^{m-1}_K(f))'_{s,x_m}=(x_m\SP^{m-1}_K(\widetilde f))'_{s,x_m}=\SP^m_{K\cup\{m\}}(\widetilde f)
\end{align*}
in $\widetilde U\setminus\R_\bullet$. 

It remains to show that $\widetilde f\in\sr(\widetilde U )$. 
We argue again as in the proof of \cite[Theorem~15]{CR_S_operators}. Since $f$ is strongly slice-regular, $\A^{m,y}_K(f)$ belongs to $\sr( U _{m+1}(y))$ for every $y\in U $ and $m\in\{0,\ldots,n-1\}$. 
Therefore $\difbar_{x_{m+1}}\widetilde\A^{m,y}_K(f)=\difbar_{x_{m+1}} \A^{m,y}_K(f)=0$ on $ U _{m+1}(y)\setminus\R$. 
From the real analyticity of $\difbar_{x_{m+1}}\widetilde\A^{m}_K(f)$ w.r.t.\ $x_{m+1}$ in $\widetilde U\setminus\R_\bullet$ and the fact that every connected component of $\widetilde  U \setminus\R_\bullet$ intersects at least one connected component of $ U \setminus\R_\bullet$, it follows that $\difbar_{x_{m+1}}\widetilde\A^{m}_K(f)=\difbar_{x_{m+1}}\SP^{m}_K(\widetilde f)=0$ on $\widetilde  U \setminus\R_\bullet$. Therefore $\SP^{m,y}_K(\widetilde f)\in\sr(\widetilde U_{m+1}(y)\setminus\R)$ for every $y\in\widetilde U\setminus\R_\bullet$. 
Proposition~\ref{pro:1varChar2} permits to conclude that $\widetilde f$ is slice-regular in $\widetilde U$.

Conversely, if $f\in\mathcal C^n(U,\hh)$ has an extension $\widetilde f\in\sr(\widetilde U)$, then Theorem \ref{teo:Almansi_several} implies that  the Almansi components $\A^m_K(\widetilde f)=\SP^m_K(\widetilde f)$ of $\widetilde f$ are slice-regular w.r.t.\ $x_{m+1}$ for every $m\in\{0,\ldots, n-1\}$. Since $\A^m_K(f)=\A^m_K(\widetilde f)_{|U}$ on $U$, $f$ is strongly slice-regular in $U$. 

The uniqueness of the extension $\widetilde f$ follows from Proposition \ref{pro:1varChar} and the identity principle for slice-regular functions of one variable (see \cite[Theorem~4.11]{AlgebraSliceFunctions}) applied to the functions $\widetilde f$ and to the Almansi components of $\widetilde f$.  
\end{proof}

If the domain is axially symmetric, the condition of strong slice-regularity is equivalent to the slice-regularity defined via the stem function approach (see \cite[Proposition~17]{CR_S_operators} for the one-variable case).

\begin{proposition}\label{pro:axially}
Let $\OO\subseteq\hh^n$ be an axially symmetric domain and $f\in\mathcal C^n(\OO,\hh)$. Then $f\in\sr(\OO)$ if and only if $f$ is strongly slice-regular in $\OO$.
Moreover, a function $f$ is strongly slice-regular in $\OO$ if and only if it is locally strongly slice-regular in $\OO$.  
\end{proposition}
\begin{proof}
If $f\in\sr(\OO)$, $\A^m_K(f)=\SP^m_K(f)$ for every $m$ and $K$ and then Theorem~\ref{teo:Almansi_several} implies that $f$ is strongly slice-regular. Conversely, if $f$ is strongly slice-regular on $\OO$, Theorem \ref{teo:extension} gives $\tilde f\in\sr(\widetilde\OO)$, with $\widetilde f=f$ on $\OO=\widetilde\OO$. Therefore $f\in\sr(\OO)$.
Since $\OO$ is axially symmetric, the sets $\s_y\cap\OO$ are connected for every $y\in\OO$. This implies the last statement. 
\end{proof}

If $U\subseteq\hh^n$ is any open set, we will denote by $\sr(U)$ the set of strongly slice-regular functions in $U$, and by  $\sr_{loc}(U)$ the set of locally strongly slice-regular functions in $U$. Thanks to Proposition \ref{pro:axially}, this notation is consistent with the usual one if $U$ is axially symmetric. 
If $U\subseteq\OO$ is open, with $\OO$ an axially symmetric domain, then the restriction on $U$ of any $f\in\sr(\OO)$ is strongly slice-regular in $U$. 

\begin{definition}\label{def:product}
Let $U\subseteq\hh^n$ be open. Given two strongly slice-regular functions $f,g\in\sr(U)$, their \emph{slice product} is the function $f\cdot g\in\sr(U)$ defined as
\[
f\cdot g:=\big(\widetilde f\cdot\widetilde g\big)_{|U},
\]
where $\widetilde f$ and $\widetilde g$ are the slice-regular extensions on $\widetilde U$ given by Theorem \ref{teo:extension}.  
\end{definition}

The previous definition, together with pointwise sum and multiplication with real numbers, provides $\sr(U)$ with a structure of real algebra. Theorem \ref{teo:extension} says that the mapping $\E$ sending $f\in\sr(U)$ to $\E(f)=\widetilde f\in\sr(\widetilde U)$ is an isomorphism of real algebras. It is also an isomorphism of right $\hh$-modules.

When $n=1$, if $f$ is \emph{slice-preserving}, i.e., $f(U\cap\cc_I)\subseteq\cc_I$ for every $I\in\s$, then 
\[
(f\cdot g)(x)=(\widetilde f\cdot\widetilde g)(x)=\widetilde f(x)\widetilde g(f(x)^{-1}xf(x))
=\widetilde f(x)\widetilde g(x)=f(x)g(x)
\]
for every $x\in U$ such that $f(x)\ne0$. Therefore it holds $(f\cdot g)(x)=f(x)g(x)$ for all $x\in U$. This property is no more valid when $n>1$. For example $x_2\cdot x_1=x_1x_2\ne x_2x_1$ (see \cite[\S2.6]{SeveralA}).

\begin{remark}\label{rem:localnotstrong}
If $ U $ is not axially symmetric, it can happen that a locally strongly slice-regular function does not extend slice-regularly to $\widetilde  U $ and therefore $\sr( U )\ne\sr_{loc}( U )$. See \cite{DouRenI,DouRenSabadini} and \cite[Example~3.4]{GS2020geometric} for a one-variable example. An example in $n$ variables can be obtained considering the product $g(x):=f(x_1)\cdot x_n$, where $f$ is a one-variable function, locally strongly slice-regular but not strongly slice-regular in $x_1$. 

If the intersections $\s_y\cap  U $ are always connected, then $\sr_{loc}( U )=\sr( U )$. For example, this is true for every open ball in $\hh^n$.
\end{remark}

\begin{theorem}[Local extendibility of locally strongly slice-regular functions]\label{teo:localext}
Let $ U \subseteq\hh^n$ be open and let $f\in\mathcal C^n( U,\hh )$. Then $f\in\sr_{loc}( U )$ if and only if for each $y\in  U $ there exists an open neighbourhood $ U _y\subseteq U$ of $y$ and a slice-regular function $\widetilde f^{(y)}\in\sr(\widetilde{ U _y})$ that extends the restriction $f_{| U _y}$.
\end{theorem}
\begin{proof}
Let $f\in\sr_{loc}(U)$. 
If $y\in U$, we can find an open neighbourhood $U_y\subseteq U$ of $y$ such that $f_{| U_y}$ is strongly slice-regular in $U_y$ 
and apply Theorem \ref{teo:extension}.  
The converse is immediate. 
\end{proof}

\begin{definition}\label{def:product_loc}
Let $U\subseteq\hh^n$ be open. Given $f,g\in\sr_{loc}(U)$, their \emph{slice product} is the function $f\cdot g\in\sr_{loc}(U)$ defined locally as
\begin{equation}\label{eq:local_product}
(f\cdot g)_{|U_y}:=\big(\widetilde f^{(y)}\cdot\widetilde g^{(y)}\big)_{|U_{y}},
\end{equation}
where $\widetilde f^{(y)}$ and $\widetilde g^{(y)}$ are local slice-regular extensions on $\widetilde U_y$ given by Theorem \ref{teo:localext}.   
\end{definition}

The product is well-defined: if two local extensions $\widetilde f^{(y)}$ and $\widetilde f^{(z)}$ of $f$ coincide on an open subset $W\subseteq U$, then $\widetilde f^{(y)}=\widetilde f^{(z)}$ on $\widetilde W$. 
Since the product $\widetilde f^{(y)}\cdot\widetilde g^{(y)}$ belongs to $\sr(\widetilde U_y)$, its restriction on $U_y$ is strongly slice-regular. Therefore formula \eqref{eq:local_product} defines a locally strongly slice-regular function on $U$. This slice product, together with pointwise sum and multiplication by reals, gives to $\sr_{loc}(U)$ a structure of real algebra. 

Thanks to Theorem \ref{teo:localext}, all \emph{local} properties satisfied by slice-regular functions are still valid for locally strongly slice-regular functions. See \cite[\S4.1]{CR_S_operators} for some results of quaternionic local slice analysis in one variable. As a several variables example, we state the following local quaternionic Almansi Theorem, which generalizes Theorem \ref{teo:Almansi_several}.  

\begin{theorem}[Local Almansi Theorem in several quaternionic variables]\label{teo:local_Almansi_several}
Let $f\in\sr_{loc}(U)$, with $U$ open set in $\hh^n$. Let $\A^m_K(f)\in\mathcal C^{n-m}(U,\hh)$ be defined as in \eqref{eq:amk} for every $m\in\{0,\ldots,n\}$ and $K\in\Pm(m)$. Then
\[
f(x)=\sum_{K\in\Pm(m)}(-\overline{x})_{K^c}\A^m_K(f)(x)
\]
for every $x\in U$. If $m=n$, then every function $\A^n_K(f)$ is separately  harmonic in the variables $x_1,\ldots, x_n$. If $m<n$, then the functions $\A^m_K(f)$ are separately harmonic in  $x_1,\ldots, x_m$ and locally strongly slice-regular w.r.t.\ $x_{m+1}$ in $U_{m+1}(y)$.
The functions $\A^m_K(f)$ are locally constant on the set $\left(\s_{y_1}\times\cdots\times\s_{y_{m}}\times\{y_{m+1}\}\times\cdots\times\{y_n\}\right)\cap U$ for every $y\in U$.
\end{theorem}
\begin{proof}
By the local extension result (Theorem \ref{teo:localext}) and the Almansi Theorem (Theorem~\ref{teo:Almansi_several}), we get that, for each $y\in  U $, it holds 
\[
f(x)=\widetilde f^{(y)}(x)=\sum_{K\in\Pm(m)}(-\overline{x})_{K^c}\SP^m_K(\widetilde f^{(y)})(x)
\]
in a neighbourhood $ U _y$ of $y$, where $\widetilde f^{(y)}\in\sr(\widetilde U_y)$ is the extension of $f_{|U_y}$. By slice-regularity and definition \eqref{eq:amk}, it holds
$\SP^m_K(\widetilde f^{(y)})=\A^m_K(\widetilde f^{(y)})=\A^m_K(f)$ on $U_y$. 
Since the functions $\SP^m_K(\widetilde f^{(y)})$ are zonal in the variables $x_1,\ldots, x_m$,  the functions $\A^m_K(f)$ are locally constant on the set $\left(\s_{y_1}\times\cdots\times\s_{y_{m}}\times\{y_{m+1}\}\times\cdots\times\{y_n\}\right)\cap U$ for every $y\in U$.
\end{proof}

\begin{remark}
For one quaternionic variable, it has been proved in \cite{GS2020local} that if $ U\subseteq\hh$ is a \emph{slice domain} (i.e., $ U\cap\R\ne\emptyset$ and every slice $ U\cap\cc_I$ is a domain in $\cc_I$), then every function $f\in\mathcal C^1( U)$ satisfying condition (i) in Definition \ref{def:loc} 
has the local slice-regular extension property, i.e., it is locally strongly slice-regular. 
In the same paper \cite{GS2020local} it has been identified a class of domains, called \emph{simple domains}, which includes convex slice domains, such that $f$ is also strongly slice-regular. 
\end{remark}

\section{Wirtinger operators}\label{sec:Wirtinger operators}

We recall the behavior of Wirtinger derivatives for functions of several complex variables. It holds 
\begin{equation}\label{eq:complex_Wirtinger}
\dd{z_j}{z_m}=\delta_{jm},\quad
\dd{\overline z_j}{z_m}=0,\quad
\dd{z_j}{\overline z_m}=0,\quad
\dd{\overline z_j}{\overline z_m}=\delta_{jm},
\end{equation}
and Leibniz rules are satisfied.
In one quaternionic variable $x$, the global differential operators $\dif$, $\difbar$ coincide with the slice derivatives $\dd{}{x}$, $\dd{\ }{x^c}$ on slice functions, and then they satisfy  formulas analogous to \eqref{eq:complex_Wirtinger} when applied to powers $x^k$ and $\overline x^k$. The same holds for several quaternionic variables for the operators $\dif_{x_1}$ and $\difbar_{x_1}$. On the contrary, the operators $\dif_{x_2},\ldots,\dif_{x_n}$ and $\difbar_{x_2},\ldots,\difbar_{x_n}$ behave differently. 

\begin{example}\label{ex:x1x2}
The operator $\dif_{x_2}$ is not a candidate for a Wirtinger operator for slice functions in two or more quaternionic variables. For example, consider the function $x_1\cdot x_2=x_1x_2\in\sr(\hh^2)$. Then $\dif_{x_1}(x_1x_2)=x_2$, but the function
\[
\textstyle\dif_{x_2}(x_1x_2)=x_{1_0}+\frac{\IM(x_2)}{|\IM(x_2)|^2}\sum_{i=1}^3x_{1_i}x_{2_i}
\]
is not even a slice function on $\hh^2\setminus\R_\bullet$, since on $\cc_i\times\cc_i$ it holds $\dif_{x_2}(x_1x_2)=x_{1_0}+\IM(x_1)=x_1$, and therefore if it were a slice function it should coincide with $x_1$ on $\hh^2\setminus\R_\bullet$. Similar properties hold in $\hh^n$ 
for any operator $\dif_{x_m}$ and $\difbar_{x_m}$ with $m>1$. 
\end{example}

For any $i,j$ with $1\le i,j\le 3$ and $m\in\{1,\ldots,n\}$, consider the
tangential differential operators 
\[
L^m_{ij}=x_{mi}\dd{\ }{x_{mj}}-x_{mj}\dd{\ }{x_{mi}}
\]
to the spheres $\s_x$. Let
\[\Gamma_{x_m}=-\tfrac12\sum_{i,j=1}^3 e_ie_jL^m_{ij}=-iL^m_{23}+jL^m_{13}-kL^m_{12}
\] 
be the quaternionic \emph{spherical Dirac operator} w.r.t.\ the variable $x_m$ (see e.g.\ 
\cite[\S8.7]{BDS} or \cite{SprossigZAA}). Here $e_1=i$, $e_2=j$, $e_3=k$. 

\begin{remark}\label{rem:Gamma}
The operator $\Gamma_{x_m}$ coincides with the operator $\Gamma_{\mathcal B}$ of \cite{CR_S_operators} when one takes $\hh$ as the hypercomplex subspace $M$, $m=3$, $c_m=-1$ and  $\mathcal B=(1,i,j,k)$.
\end{remark}

Let $f\in\mathcal C^n(  U ,\hh)$, with $  U \subseteq\hh^n$ open. Define recursively 
functions $\Gamma^m_K(f)$, with $m\in\{0,\ldots, n\}$ and $K\in\Pm(m)$, as follows: 
\begin{equation}\label{eq:Gamma_mk}
\begin{cases}
\Gamma^0_\emptyset(f):=f&\text{\quad for }m=0,\\
\Gamma^m_K(f):=(2\IM(x_m))^{-1}\Gamma_{x_m}\left(x_m^{\mathbf{1}_K(m)}\Gamma^{m-1}_{K\setminus\{m\}}(f)\right) &\text{\quad for }m\in\{1,\ldots,n\}.
\end{cases}
\end{equation}
For $m>0$, the functions $\Gamma^m_K(f)$ are of class $\mathcal C^{n-m}(U\setminus\R_m,\hh)$.

\begin{definition}[Quaternionic Wirtinger operators]\label{def:Wirtinger}
Let $U\subseteq\hh^n$ be an open set. For every $m\in\{1,\ldots,n\}$, let $\thet m,\thb m:\mathcal C^n(U\setminus\R_\bullet,\hh)\to\mathcal C^{n-m}(U\setminus\R_\bullet,\hh)$ be the $\R$-linear differential operators of order $m$ defined on functions $f\in\mathcal C^n(U\setminus\R_\bullet,\hh)$ by 
\[
\theta_m(f):=\sum_{K\in\Pm(m-1)}(-\overline{x})_{K^c}\dif_{x_m}\Gamma^{m-1}_K(f)\quad\text{and}\quad
\thb m(f):=\sum_{K\in\Pm(m-1)}(-\overline{x})_{K^c}\difbar_{x_m}\Gamma^{m-1}_K(f),
\]
where $\dif_{x_m}$ and $\difbar_{x_m}$ are the operators defined in \eqref{def:thb}.  
\end{definition}

Observe, in particular, that the first operators $\thet 1$ and $\thb 1$ coincide, respectively, with $\dif_{x_1}$ and $\difbar_{x_1}$, while the second operators have the form
\begin{align*}
\thet 2(f)&=\dif_{x_2}\big((2\IM(x_1))^{-1}\Gamma_{x_1}(x_1 f)\big)-\overline x_1\dif_{x_2}\big((2\IM(x_1))^{-1}\Gamma_{x_1}f\big),\\
\overline\theta_2(f)&=\difbar_{x_2}\big((2\IM(x_1))^{-1}\Gamma_{x_1}(x_1 f)\big)-\overline x_1\difbar_{x_2}\big((2\IM(x_1))^{-1}\Gamma_{x_1}f\big).
\end{align*}
Note also that the operators $\thet m$ and $\thb m$ are of total order $m$, but of the first order w.r.t.\ any variable $x_1,\ldots,x_m$.

Let  $\OO'\subseteq\hh$. We recall that on a slice function $g\in\SF^1(\OO')$ in one variable $x_m$, the operators $\dif_{x_m}$ and $\difbar_{x_m}$ act as the slice derivatives: $\dif_{x_m}(g)=\dd{g\;}{x_m}$ and $\difbar_{x_m}(g)=\dd{g\;}{x_m^c}$ on $\OO'\setminus\R$. Moreover, $g\in\sr(\OO')$ if and only if $\difbar_{x_m}(g)=0$ (see \cite[Theorem~2.2]{Gh_Pe_GlobDiff}).

\begin{example}
Consider the function $x_1x_2\in\sr(\hh^2)$ of Example \ref{ex:x1x2}. It holds $\thet 1(x_1x_2)=\dif_{x_1}(x_1x_2)=x_2$, and
$\Gamma_{x_1}(x_1x_2)=2\IM(x_1) x_2$, $\Gamma_{x_1}(x_1^2x_2)=4\IM(x_1)x_{1_0} x_2$, from which we deduce
\[
\thet{2}(x_1x_2)=\dif_{x_2}(2x_{1_0}x_2)-\overline x_1\dif_{x_2}(x_2)=2x_{1_0}-\overline x_1=x_1,
\]
while $\thb 1(x_1x_2)=\thb2(x_1x_2)=0$.
\end{example}

\begin{remark}\label{rem:slice_derivatives}
When $f\in\SF^1(\OO)$ is a slice function in $\OO\subseteq\hh^n$ and also a slice function w.r.t.\ one of the variables, say $x_h$, then one can define both the slice partial derivative $\dd{f}{x_h}$ and the slice derivative w.r.t.\ the variable $x_h$, that here we denote provisionally by $\left(\dd{}{x_h}\right)_1(f)$. 
We recall from \cite[Proposition 4.1]{Binosi1} that a slice function $f=\I(F)=\I(\sum_Ke_KF_K)\in\SF(\OO)$ is also a slice function w.r.t.\ $x_h$ if and only if $F_{K\cup\{h\}}=0$ for every $K\in\Pm(n)$ such that $K\cap\{1,\ldots,h-1\}\ne\emptyset$. 
A direct computation shows that for every $y\in\OO$ and $x_h\in\OO_h(y)$, it holds 
\[
\dd{f}{x_h}(y_1,\ldots,y_{h-1},x_h,y_{h+1},\ldots,y_n)
=\dd{f_h^{(y)}}{x_h}(y_1,\ldots,y_{h-1},x_h,y_{h+1},\ldots,y_n),
\]
where as above $f_h^{(y)}$ denotes the function $x_h\mapsto f(y_1,\ldots,y_{h-1},x_h,y_{h+1},\ldots,y_n)$. 
Since there is no ambiguity, we will denote both slice derivatives, $\dd{f}{x_h}$ and $\left(\dd{}{x_h}\right)_1(f)$, with the same symbol $\dd{f}{x_h}$. The same holds for the slice partial derivative $\dd{f}{x_h^c}$.
\end{remark}

The results of the next proposition justify the name \emph{quaternionic Wirtinger operators} given to $\thet m$ and $\thb m$.

\begin{proposition}\label{pro:powers}
For every multiindex $\ell=(\ell_1,\ldots,\ell_n)\in\nn^n$, let $x^\ell=x_1^{\ell_1}\cdots x_n^{\ell_n}$ and $\overline x^\ell=\overline x_1^{\ell_1}\cdots \overline x_n^{\ell_n}$. Let $m\in\{1,\ldots,n\}$ and $e_m=(0,\ldots,1,\ldots,0)\in\nn^n$ with $1$ only in position $m$. Then
\begin{align*}
\thet m(x^\ell)&=\dd {x^\ell}{x_m}=\ell_m x^{\ell-e_m}\quad\text{and}\quad
\thb m(x^\ell)=\dd {x^\ell}{x_m^c}=0;\\
\thet m(\overline x^\ell)&=\dd {\overline x^\ell}{x_m}=0\quad\text{and}\quad
\thb m(\overline x^\ell)=\dd {\overline x^\ell}{x_m^c}=\ell_m \overline x^{\ell-e_m}.
\end{align*}
Here $\dd {\ }{x_m}$ and $\dd {\ }{x_m^c}$ are the slice partial derivatives (see \cite[Definition~3.10]{SeveralA} and \S\ref{sec:sev_var}). 
\qed
\end{proposition}

Instead of giving a direct proof of the proposition, we deduce it from a more general result (Theorem \ref{teo:Wirtinger}) valid for every slice function $f\in\SF(\OO)$, with $\OO\subseteq\hh^n$ axially symmetric. Before stating this result, we need the following consequence of Theorem \ref{teo:slice_Almansi_several}.

\begin{proposition}\label{pro:slice_Gamma_several}
If $f\in\SF(\OO)\cap\mathcal C^n(\OO,\hh)$, with $\OO\subseteq\hh^n$ an axially symmetric open set, then $\Gamma^{m}_K(f)\in\SF(\OO\setminus\R_m)$ and for $m<n$ the functions $\Gamma^{m}_K(f)$ are slice functions w.r.t.\  $x_{m+1}$. Moreover, for every $x\in\OO\setminus\R_m$ it holds 
\begin{equation}\label{eq:slice_Gamma_several}
f(x)=\sum_{K\in\Pm(m)}(-\overline{x})_{K^c}\Gamma^m_K(f)(x)=\sum_{K\in\Pm(m)}(-\overline{x})_{K^c}\cdot\Gamma^m_K(f)(x).
\end{equation}
\end{proposition}
\begin{proof}
From Proposition~\ref{pro:1varChar_slice}, Remark \ref{rem:Gamma} and \cite[Proposition~9(a)]{CR_S_operators}, it follows that when $f$ is a slice function, then it holds $\Gamma^m_K(f)=\SP^m_K(f)$ for every $m$ and $K$.
Then \eqref{eq:slice_Gamma_several} is equivalent to formula \eqref{eq:slice_Almansi_several} of Theorem \ref{teo:slice_Almansi_several}.
\end{proof}

\begin{theorem}\label{teo:Wirtinger}
Let $\OO\subseteq\hh^n$ be an axially symmetric open set. Let $f\in\SF(\OO)\cap\mathcal C^n(\OO\setminus\R_\bullet,\hh)$. Then it holds on $\OO\setminus\R_\bullet$:
\begin{equation}\label{eq:theta_slice}
\thet m(f)=\dd {f}{x_m}\quad\text{and}\quad \thb m(f)=\dd {f}{x_m^c}\text{\quad for every }m\in\{1,\ldots,n\}.
\end{equation}
The Wirtinger operators $\thet1,\ldots,\thet n$, $\thb1,\ldots,\thb n:\SF(\OO)\cap\mathcal C^n(\OO\setminus\R_\bullet,\hh)\to \SF(\OO\setminus\R_\bullet)\cap\mathcal C^0(\OO\setminus\R_\bullet,\hh)$ commute with each other, and for any $f,g\in\SF(\OO)\cap\mathcal C^n(\OO\setminus\R_\bullet,\hh)$, they satisfy the Leibniz rules
\[
\thet m(f\cdot g)=\thet m(f)\cdot g+f\cdot\thet m(g)\text{\quad and\quad}\thb m(f\cdot g)=\thb m(f)\cdot g+f\cdot\thb m(g)
\]
on $\OO\setminus\R_\bullet$. 
\end{theorem}
\begin{proof}
The case $m=1$ is immediate: since $f$ a slice function w.r.t.\ $x_1$, it follows that $\thb{1}(f)=\difbar_{x_1}(f)=\dd{f}{x_1^c}$ and $\thet1(f)=\dif_{x_1}(f)=\dd{f}{x_1}$ (recall also Remark \ref{rem:slice_derivatives}).
Let $m>1$. From Proposition \ref{pro:slice_Gamma_several} we get
\[
f(x)=\sum_{K\in\Pm(m-1)}(-\overline{x})_{K^c}\cdot\Gamma^{m-1}_K(f)(x).
\]
From the Leibniz formula for slice derivatives w.r.t.\ the slice product \cite[Proposition~3.25]{SeveralA} and additivity of the slice partial derivatives, we get
\begin{align}\label{eq:ddtheta}
\dd{f}{x_m}&=\sum_{K\in\Pm(m-1)}\dd{}{x_m}\left((-\overline{x})_{K^c}\cdot\Gamma^{m-1}_K(f)\right)=
\sum_{K\in\Pm(m-1)}(-\overline{x})_{K^c}\dd{\Gamma^{m-1}_K(f)}{x_m}=\\\notag
&=\sum_{K\in\Pm(m-1)}(-\overline{x})_{K^c}\dif_{x_m}{\Gamma^{m-1}_K(f)}=\thet m(f).
\end{align}
To obtain the second equality in \eqref{eq:ddtheta}, we used the fact that $K^c\in\Pm(m-1)$ and then $\dd{(-\overline{x})_{K^c}}{x_m}$ vanishes.  To obtain the third equality, we used Remark \ref{rem:slice_derivatives}. 
Along the same lines, using the fact that $\dd{g}{x_m^c}=\difbar_{x_m}{(g)}$ for every slice function w.r.t.\ $x_m$, one obtains the second equality in \eqref{eq:theta_slice}. From formula \eqref{eq:theta_slice} it follows that the Wirtinger operators commute with each other, since this is true for the slice partial derivatives (see formula (46) in \cite{SeveralA}). Leibniz rules follows immediately from \eqref{eq:theta_slice} and \cite[Proposition~3.25]{SeveralA}.  
\end{proof}

Now we apply Wirtinger operators to study the form of slice-regular homogeneous polynomials in the $4n$ real components of $x\in\hh^n$. 

\begin{proposition}\label{pro:independence}
Let $\OO\subseteq\hh^n$ be an axially symmetric open set and let $f\in\SF(\OO)\cap\mathcal C^n(\OO\setminus\R_\bullet,\hh)$. If $f$ depends only on the last $n-m+1$ variables $x_{m},\ldots,x_n$, then 
\[
\thet h(f)=\thb h(f)=0\,\text{ for }h=1\ldots,m-1\quad\text{and}\quad \thet m(f)=\dif_{x_m}(f),\ \thb m(f)=\difbar_{x_m}(f).
\]
\end{proposition}
\begin{proof}
We can assume $m\ge2$. We prove inductively on $h=2,\ldots,m$ that
\be\label{eq:induction}
\Gamma^{h-1}_K(f)=
\begin{cases}
f&\text{if $K=\{1,\ldots,h-1\}$},\\
0&\text{if $K\in\Pm(h-1)$, $K\ne\{1,\ldots,h-1\}$}.
\end{cases}
\eeq
If $h=2$, since $f$ does not depend on $x_1$, it holds $\Gamma^1_\emptyset(f)=(2\IM(x_1))^{-1}\Gamma_{x_1}(f)=0$,  and, thanks to \cite[Proposition~9(a)]{CR_S_operators}, $\Gamma^1_{\{1\}}(f)=(2\IM(x_1))^{-1}\Gamma_{x_1}(x_1f)=(x_1f)'_{s,x_1}=(f)^\circ_{s,x_1}=f$. 
Assume that \eqref{eq:induction} holds for the value $h-2$. Then 
\[
\Gamma^{h-1}_K(f)=
\begin{cases}
(2\IM(x_{h-1}))^{-1}\Gamma_{x_{h-1}}(x_{h-1}f)=(x_{h-1}f)'_{s,x_{h-1}}=f&\text{if $K=\{1,\ldots,h-1\}$},\\
(2\IM(x_{h-1}))^{-1}\Gamma_{x_{h-1}}(\Gamma^{h-2}_{K\setminus\{h-1\}}f)=0&\text{if $K\ne\{1,\ldots,h-1\}$},
\end{cases}
\]
and \eqref{eq:induction} is proved for every $h=2,\ldots,m$. 

If $h<m$, then $\dif_h(\Gamma^{h-1}_K(f))=\difbar_h(\Gamma^{h-1}_K(f))=0$, since $\Gamma^{h-1}_K(f)$ depends only on the variables $x_m,\ldots,x_n$. If $h=m$, in the sums defining $\thet m(f)$ and $\thb m(f)$ only the terms with $K=\{1,\ldots,m-1\}$ survive, and then
$\thet m(f)=\dif_{x_m}(f)$, $\thb m(f)=\difbar_{x_m}(f)$.
\end{proof}

\begin{corollary}
Assume that the four real components of a function $f:\hh^n\simeq\R^{4n}\to\hh\simeq\R^4$ are homogeneous polynomials of degree $d_m$ in the real variables $x_{m_0},x_{m_1},x_{m_2},x_{m_3}$, for each $m=1,\ldots,n$. Then $f$ is slice-regular if and only if it is of the form 
\[
f(x_1,\ldots,x_n)=x_1^{d_1}\cdots x_n^{d_n}a,\text{\quad with $a\in\hh$ constant}.
\]
\end{corollary}
\begin{proof}
The one-variable case was proved in \cite[\S2.2]{Gh_Pe_GlobDiff}. If $f$ is slice-regular, then it is slice-regular on $\hh$ w.r.t.\ $x_1$. We then get that $f=x_1^{d_1}g$, with $g$ independent of $x_1$. By Leibniz rule and Proposition \ref{pro:independence} we get $0=\thb2(f)=x_1^{d_1}\thb2(g)=x_1^{d_1}\difbar_{x_2}(g)$, i.e.,  $g$ is slice-regular w.r.t.\ $x_2\in\hh$. Since $g$ is homogeneous of degree $d_2$ in $x_{2_0},x_{2_1},x_{2_2},x_{2_3}$, we obtain as above that $g(x_2,\ldots,x_n)=x_2^{d_2}h(x_3,\ldots,x_n)$. Iterating this argument we get the result.   
\end{proof}

Quaternionic Wirtinger operators satisfy also a conjugation property. Given a slice function $f=\I(F)$ on $\OO\subseteq\hh^n$, with $F=\sum_Ke_KF_K$, let $\widetilde F=\sum_Ke_K(-1)^{|K|}F_K$ and let
\begin{equation}\label{eq:conjugate_slice}
\overline f:=\I(\widetilde F)
\end{equation}
be the \emph{conjugate slice function} associated with $f$. 
Clearly $\overline f\in\SF(\OO)$ and $\overline{\overline f}=f$. Moreover, $(\overline{f\cdot g)}=\overline f\cdot\overline g$ for every pair of slice functions $f,g\in\SF(\OO)$. 

If $\overline f\in\sr(\OO)$, i.e., $\dibar_h\widetilde F=0$ for every $h=1,\ldots,n$, then $\partial_h F=0$ for every $h=1,\ldots,n$ ($F$ is anti-holomorphic), and vice versa (see \cite[Lemma~3.9]{SeveralA}). In this case $f$ is called \emph{conjugate slice-regular}, 
and it holds $\dd{f}{x_h}=0$ for every $h=1,\ldots,n$. 
Therefore a slice function $f$ is slice-regular if and only if the conjugate $\overline f$ is conjugate slice-regular. 
For example, every product\, $\overline x^\ell=\overline x_1^{\ell_1}\cdots \overline x_n^{\ell_n}=\overline{x^\ell}$ with $\ell\in\nn^n$, is conjugate slice-regular on $\hh^n$. 

\begin{proposition}\label{pro:Wirtinger_conjugate}
{}Let $\OO\subseteq\hh^n$ be an axially symmetric open set. Let $f\in\SF(\OO)\cap\mathcal C^n(\OO\setminus\R_\bullet,\hh)$. Then it holds on $\OO\setminus\R_\bullet$:
\begin{equation}\label{eq:Wirt_conj}
\overline{\thet m(f)}=\thb m(\overline f)\quad\text{and}\quad \overline{\thb m(f)}=\thet m(\overline f)\text{\quad for every }m\in\{1,\ldots,n\}.
\end{equation}
\end{proposition}
\begin{proof}
It is sufficient to prove the first equality, since the other follows from the first by conjugation. Let $f=\I(F)$ and $\overline f=\I(\widetilde F)$ be as above and let $\OO=\OO_D$. Then $F$ is of class $\mathcal C^n$ in $D\setminus\R$. 
From Theorem \ref{teo:Wirtinger} we deduce the equivalences 
\[
\overline{\thet m(f)}=\thb m(\overline f)\Leftrightarrow \overline{\dd{f}{x_m}}=\dd{\overline f}{x_m^c}\Leftrightarrow \overline{\I(\partial_mF)}=\I(\dibar_m\widetilde F)\Leftrightarrow \widetilde{\partial_m F}=\dibar_m\widetilde F\text{\quad in }D\setminus\rr.
\]
From \cite[Lemma~3.9]{SeveralA} we get that the $K$-components of $\widetilde{\partial_m F}$ and of $\dibar_m\widetilde F$ are given by
\begin{align*}
(\widetilde{\partial_m F})_K &=\frac12(-1)^{|K|}\left(\dd{F_K}{\alpha_m}-
\dd{F_{K\Delta\{m\}}}{\beta_m}(-1)^{|K\cap\{m\}|}\right)\text{\quad and}
\\
(\dibar_m\widetilde F)_K&=\frac12\left((-1)^{|K|}\dd{F_K}{\alpha_m}+(-1)^{|K\Delta\{m\}|}
\dd{F_{K\Delta\{m\}}}{\beta_m}(-1)^{|K\cap\{m\}|}\right).
\end{align*}
Here $(\alpha_1+i\beta_1,\ldots,\alpha_n+i\beta_n)$ are the coordinates of $D$, and $K\Delta\{m\}$ denotes the symmetric difference of the sets $K$ and $\{m\}$. 
Since $|K|$ and $|K\Delta\{m\}|$ have different parity, $(\widetilde{\partial_m F})_K=(\dibar_m\widetilde F)_K$ for every $K$ and the statement is proved.
\end{proof}

The quaternionic Wirtinger operators operators $\thb 1,\ldots,\thb n$ characterize slice-regular functions in the class of slice functions. This result extends to several variables the one-variable characterization obtained in \cite[Theorems 2.2 and 2.4]{Gh_Pe_GlobDiff}.  

\begin{theorem}\label{teo:Wirtinger_regular}
Let $\OO\subseteq\hh^n$ be an axially symmetric open set. Let $f\in\SF(\OO)\cap\mathcal C^0(\OO,\hh)\cap\mathcal C^n(\OO\setminus\R_\bullet,\hh)$.  
Then $f\in\sr(\OO)$ if and only if $\thb m(f)=0$ for every $m\in\{1,\ldots,n\}$.
\end{theorem}
\begin{proof}
If $f\in\sr(\OO)$, Theorem \ref{teo:Wirtinger} gives $\thb m(f)=\dd{f}{x_m^c}=0$ on $\OO\setminus\R_\bullet$ for every $m\in\{1,\ldots,n\}$. Conversely, assume that $\thb m(f)=0$ for every $m$. Then $\thb 1(f)=\difbar_{x_1}(f)=0$, i.e., $f$ is slice-regular w.r.t.\ $x_1$, and
\be\label{eq:zero}
\thb m(f)=\sum_{K\in\Pm(m-1)}(-\overline{x})_{K^c}\difbar_{x_m}\Gamma^{m-1}_K(f)=\sum_{K\in\Pm(m-1)}(-\overline{x})_{K^c}\difbar_{x_m}\SP^{m-1}_K(f)=0
\eeq
for $m\in\{2,\ldots,n\}$. We show that the functions $g_K:=\difbar_{x_m}\SP^{m-1}_K(f)=\dd{\SP^{m-1}_K(f)}{x_m^c}$ are constant on the set $\s_{y_1}\times\cdots\times\s_{y_{m-1}}\times\{y_{m}\}\times\cdots\times\{y_n\}$ for every $y\in\OO\setminus\R_\bullet$. 
From Proposition \ref{pro:SP_slice} it follows that $\SP^{m-1}_K(f)$ is induced by a stem function $G^{m-1}_K$ of the form $\textstyle\sum_{H\in\Pm(n),H_{m-1}=\emptyset}e_HG_H$, where $H_{m-1}=H\cap\{1,\ldots,m-1\}$. 
In view of Remark \ref{rem:slice_derivatives} and \cite[Lemma 3.9]{SeveralA}, the function $g_K$ is induced by the stem function in $n$ variables
\begin{align}\label{eq:dbargk}
&\dibar_m\left(
\sum_{H\in\Pm(n),H_{m-1}=\emptyset} e_HG_H\right)=
\sum_{H\in\Pm(n),H_{m-1}=\emptyset} \left(e_H\dd{G_H}{\alpha_m}+\I_m(e_H)\dd{G_H}{\beta_m}\right)=\\
&=
\sum_{H\ni m,H_{m-1}=\emptyset} \left(e_H\dd{G_H}{\alpha_m}-e_{H\setminus\{m\}}\dd{G_H}{\beta_m}\right)+
\sum_{H\not\ni m,H_{m-1}=\emptyset} \left(e_H\dd{G_H}{\alpha_m}+e_{H\cup\{m\}}\dd{G_H}{\beta_m}\right).\notag
\end{align}
Since $(H\Delta\{m\})\cap\{1,\ldots,m-1\}=\emptyset$ for every $H\in\Pm(n)$ with $H_{m-1}=\emptyset$, the $K'$-components of the stem function inducing $g_K$ are zero for every $K'\ne\emptyset$ such that $K'\cap\{1,\ldots,m-1\}=\emptyset$. We infer that $g_K$ is constant on the set $\s_{y_1}\times\cdots\times\s_{y_{m-1}}\times\{y_{m}\}\times\cdots\times\{y_n\}$. 
From \eqref{eq:zero} and Proposition \ref{pro:converse_slice_Almansi_several} applied in $\OO\setminus\R_\bullet$, we obtain that $g_K=\difbar_{x_m}\SP^{m-1}_K(f)=0$ for every $K\in\Pm(m-1)$ and every $m=2,\ldots,n$. Therefore $\SP^{m-1}_K(f)$ is slice-regular w.r.t.\ $x_m$, i.e., $\SP^{m-1,y}_K(f)\in\sr(\OO_{m}(y)\setminus\R)$ for every $y\in\OO\setminus\R_\bullet$. 
Thanks to Proposition \ref{pro:1varChar2}, we can conclude that $f\in\sr(\OO)$. 
\end{proof}

\begin{corollary}\label{cor:loc_Wirtinger_regular}
If $f\in\sr_{loc}(U)$, with $U$ open in $\hh^n$, then $\thb m(f)=0$ for every $m\in\{1,\ldots,n\}$. 
\end{corollary}
\begin{proof}
The thesis follows from the local extendibility theorem (Theorem \ref{teo:localext}) and Theorem \ref{teo:Wirtinger_regular}. 
\end{proof}

In the next section we will investigate the validity of the converse of Corollary \ref{cor:loc_Wirtinger_regular} on an arbitrary open subset of $\hh^n$.

\begin{remark}
If $\OO=\OO_D$, with $D\subseteq\C^n$ connected and intersecting $\R^n$, then the assumptions on $f$ in Theorem \ref{teo:Wirtinger_regular} can be relaxed, since the sliceness of $f$ is a consequence of condition $\thb m(f)=0$ for every $m\in\{1,\ldots,n\}$. This can be shown using the analogous one-variable result proved in \cite[Theorem 2.4]{Gh_Pe_GlobDiff}. 
\end{remark}

\section{Local slice functions of one and several variables}\label{sec:Local_slice_functions}

We begin this section showing that the decomposition of Theorem \ref{teo:slice_Almansi_several} can be extended to any sufficiently smooth function (not necessarily slice) $f$ on any open subset of $\hh^n$. Here the functions $\SP^m_K(f)$ 
will be replaced by the functions $\Gamma^m_K(f)$ obtained by applying the differential operators $\Gamma_{x_m}$. 

\begin{theorem}
\label{teo:no_slice_Almansi_several}
Let $f\in\mathcal C^n(U\setminus\R_\bullet,\hh)$ be a function defined on an open set $U\subseteq\hh^n$. Then for every fixed $m\in\{1,\ldots, n\}$ it holds  
\begin{equation}\label{eq:no_slice_Gamma_several}
f(x)=\sum_{K\in\Pm(m)}(-\overline{x})_{K^c}\Gamma^m_K(f)(x)
\end{equation}
for every $x\in U\setminus\R_\bullet$. 
\end{theorem}
\begin{proof}
The proof is similar to the first part of the one given for 
 Theorem \ref{teo:slice_Almansi_several}. We proceed by induction on $m$. 
When $m=1$, Proposition 7 in \cite{CR_S_operators}  gives the equality
\be\label{eq:g1}
f(x)=(2\IM(x_1))^{-1}\left(\Gamma_{x_1}(x_1f(x))-\overline x_1\,\Gamma_{x_1}(f(x))\right)=\Gamma^1_{\{1\}}(f)(x)-\overline x_1\,\Gamma^1_\emptyset(f)(x)
\eeq
for every $x\in U\setminus\R_\bullet$. Assume that formula \eqref{eq:no_slice_Gamma_several} holds for an integer $m<n$. If $K\in\Pm(m)$, applying again \cite[Proposition~7]{CR_S_operators} we get
\begin{align}\label{eq:gammamK1}
\Gamma^m_K(f)&=(2\IM(x_{m+1}))^{-1}\left(\Gamma_{x_{m+1}}(x_{m+1}\Gamma^m_K(f))-\overline x_{m+1}\,\Gamma_{x_{m+1}}(\Gamma^m_K(f))\right)=\\\notag
&=\Gamma^{m+1}_{K\cup\{m+1\}}(f)-\overline x_{m+1}\,\Gamma^{m+1}_K(f)
\end{align}
on $U \setminus\R_\bullet$.
By the induction hypothesis and \eqref{eq:gammamK1},
\begin{align*}
f(x)&=
\sum_{K\in\Pm(m)}(-\overline{x})_{K^c}\Gamma^m_K(f)(x)=\\
&=
\sum_{K\in\Pm(m)}(-\overline{x})_{K^c}\Gamma^{m+1}_{K\cup\{m+1\}}(f)(x)-\sum_{K\in\Pm(m)}(-\overline{x})_{K^c}\overline x_{m+1}\,\Gamma^{m+1}_K(f)(x)=\\
&=
\sum_{H\in\Pm(m+1)}(-\overline{x})_{H^c}\Gamma^{m+1}_H(f)(x)
\end{align*}
for every $x\in U\setminus\R_\bullet$. 
\end{proof}

The previous theorem suggests a way to extend Definitions \ref{def:loc} and \ref{def:loc_several} to introduce the concepts of strongly and locally strongly slice function. We begin with the one-dimensional case. For $i,j\in\{1,2,3\}$, let  
$
\textstyle L_{ij}=x_{i}\dd{\ }{x_{j}}-x_{j}\dd{\ }{x_{i}}$ and $\Gamma_{x}=-iL_{23}+jL_{13}-kL_{12}.
$ 

\begin{definition}\label{def:loc_slice}
Let $  U \subseteq \hh$ be any open set and let $f\in\mathcal C^0(U,\hh)\cap\mathcal C^1(U\setminus\R,\hh )$. The function $f$ is called \emph{locally strongly slice}, and we write $f\in\SF^+_{loc}( U )$, if the functions $(\IM(x))^{-1}\Gamma_{x}f$ and $(\IM(x))^{-1}\Gamma_{x}(xf)$ 
 are locally constant on $\s_y\cap   U $ for every $y\in U\setminus\R $. 

If the functions $(\IM(x))^{-1}\Gamma_{x}f$ and $(\IM(x))^{-1}\Gamma_{x}(xf)$ are constant on $\s_y\cap   U $ for every $y\in U\setminus\R$, then $f$ is called \emph{strongly slice}, and we write $f\in\SF^+( U )$. 
\end{definition}

If $U=\OO_D$ is axially symmetric, then the two definitions of locally strongly slice and strongly slice function are equivalent and coincide with the usual one of slice function with some smoothness. 
Indeed, if $f\in\SF^0(\OO_D)\cap\mathcal C^1(\OO_D\setminus\R)$ is a slice function induced by a continuous stem function, then $f\in\mathcal C^0(\OO_D,\hh)$ (see \cite[Theorem 2.26]{SeveralA}), and it holds $\Gamma_{x}f=2\IM(x)\sd f$ and $\Gamma_{x}(xf)=2\IM(x)\sd{(xf)}$  on $\OO_D\setminus\R$ (see \cite[Proposition~9(a)]{CR_S_operators}), and therefore $\Gamma_{x}f$ and $\Gamma_{x}(xf)$ are constant on $\s_y$ for every $y\in \OO_D\setminus\R $. 
Conversely, if $f$ is strongly slice in $\OO_D$, then Theorem \ref{teo:no_slice_Almansi_several} gives $f=(2\IM(x))^{-1}\left(\Gamma_{x}(xf)-\overline x\,\Gamma_{x}(f)\right)$ in $\OO_D\setminus\R$. Since the functions  $(2\IM(x))^{-1}\Gamma_{x}(xf)$ and $(2\IM(x))^{-1}\Gamma_{x}(f)$ are circular slice functions in $\OO_D\setminus\R$, also $f$ is a slice function in $\OO_D\setminus\R$ and then, being continuous on $\OO_D$, it is a slice function in $\OO_D$ (see the following Remark \ref{rem:extR}). 

Observe that when $U\subseteq\hh$ is a slice domain, 
every function in $\SF^+_{loc}(U)$ is also a \emph{locally slice function} in the sense of \cite[Definition 3.6]{GS2020geometric}.

\begin{remark}\label{rem:PDE}
Since the operators $L_{ij}$ ($i,j=1,2,3$) generate the space of tangential differential operators to the 2-spheres $\s_x$, a function $f\in\mathcal C^0(U,\hh)\cap\mathcal C^2(U\setminus\R,\hh)$ is locally strongly slice if and only if it satisfies the system of second order partial differential equations
\begin{equation}\label{eq:system}
L_{ij}\big((\IM(x))^{-1}\Gamma_x f\big)=L_{ij}\big((\IM(x))^{-1}\Gamma_x (xf)\big)=0\quad\forall\, i,j\in\{1,2,3\}
\end{equation}
in $U\setminus\R$.
\end{remark}

\begin{remark}\label{rem:extR}
If $f\in\SF(\OO_D\setminus\R)\cap\mathcal C^0(\OO_D)$, then $f\in\SF(\OO_D)$. More precisely, $\SF(\OO_D\setminus\R)\cap\mathcal C^0(\OO_D)=\SF^0(\OO_D)$. Indeed, let $f=\I(F)=\I(F_1+iF_2)$, with $F:D\setminus\R\to \hh\otimes_\R\cc$. If $z=\alpha+i\beta\in D\setminus\R$ and $x=\alpha+I\beta\in\OO_D\setminus\R$, with $I\in\s$, then $\lim_{z\to\alpha}F_1(z)=\lim_{z\to\alpha}(f(x)+f(\overline x))/2=f(\alpha)$. Therefore
\[
\|F_2(z)\|=\|IF_2(z)\|=\|f(x)-F_1(z)\|\le\|f(x)-f(\alpha)\|+\|f(\alpha)-F_1(z)\|\to0
\]
as $z$ tends to $\alpha\in\R$. Setting $F_1(\alpha):=f(\alpha)$ and $F_2(\alpha):=0$, one obtains a continuous stem function on $D$ which induces $f$ on $\OO_D$.
This property extends to any number of variables thanks to formula (8) in \cite[Proposition 2.12]{SeveralA}. If $f=\I(F)\in\SF(\OO_D\setminus\R_\bullet)\cap\mathcal C^0(\OO_D)$, then for every $x=(\alpha_1+I_1\beta_1,\ldots,\alpha_n+I_n\beta_n)\in\OO_D\setminus\R_\bullet$ and every $K=\{k_1,\ldots,k_p\}\in\Pm(n)$,  it holds
\be\label{eq:sliceness}
F_K(z)=2^{-n}I_{k_p}^{-1}\cdots I_{k_1}^{-1}
\sum_{H\in\Pm(n)}(-1)^{|K\cap H|}f(x^{c,H}),
\eeq 
where $z=(\alpha_1+i\beta_1,\ldots,\alpha_n+i\beta_n)\in D$ and $x^{c,H}\in\OO_D\setminus\R_\bullet$ is obtained by $x$ conjugating all the components of indices $h_j\in H$. Formula \eqref{eq:sliceness} and the continuity of $f$ in $\OO_D$ shows that every $F_K$ extends continuously to $D$. The sliceness criterion of \cite[Corollary 2.16]{SeveralA} implies that the function $F=\sum_{K\in\Pm(n)}e_KF_K$ is a continuous slice function on $D$. Therefore 
$\SF(\OO_D\setminus\R_\bullet)\cap\mathcal C^0(\OO_D)=\SF^0(\OO_D)$.
\end{remark}

\begin{definition}\label{def:loc_slice_several}
Let $ U \subseteq \hh^n$, with $n>1$, be any open set and let $f\in\mathcal C^0(U,\hh)\cap\mathcal C^n( U\setminus\R_\bullet,\hh)$. 
The function $f$ is called \emph{locally strongly slice}, 
and we write $f\in\SF^+_{loc}( U )$, if for every $y\in U\setminus\R_\bullet $, every $m\in\{0,\ldots, n-1\}$ and every $K\in\Pm(m)$, the function $(\Gamma^{m}_K(f))^{(y)}$ belongs to $\SF^+_{loc}( U _{m+1}(y))$. 

If it holds that for every $y\in U\setminus\R_\bullet$, every $m\in\{0,\ldots, n-1\}$ and every $K\in\Pm(m)$, the function $(\Gamma^{m}_K(f))^{(y)}$ belongs to $\SF^+( U _{m+1}(y))$, then  $f$ is called \emph{strongly slice}. In particular, every function that is strongly slice in $U$ belongs to $\SF^+_{loc}( U )$. 
\end{definition}

In view of Definition \ref{def:loc_slice} and of formula \eqref{eq:Gamma_mk} defining $\Gamma^m_K(f)$, we can rephrase the previous definition in the following form. 

\begin{proposition}\label{pro:loc_slice_several}
Let $ U \subseteq \hh^n$, with $n>1$, be any open set and let $f\in\mathcal C^0(U,\hh)\cap\mathcal C^n( U\setminus\R_\bullet,\hh)$. The function $f$ is locally strongly 
slice in $U$ if and only if for every $y\in U\setminus\R_\bullet$, 
and for every $m\in\{1,\ldots, n\}$ and $K\in\Pm(m)$, the functions $\Gamma^{m}_{K}(f)$ are locally constant on $(\s_{y_1}\times\cdots\times\s_{y_{m}}\times\{y_{m+1}\}\times\cdots\times\{y_n\})\cap U $.

The function $f$ is strongly slice if and only if the functions $\Gamma^{m}_K(f)$ are constant on the set $(\s_{y_1}\times\cdots\times\s_{y_{m}}\times\{y_{m+1}\}\times\cdots\times\{y_n\})\cap U $. \qed
\end{proposition}

\begin{proposition}
Let $ U \subseteq\hh^n$ be open. 
If $f$ is strongly slice-regular in $U$, then  $f$ is strongly slice in $U$, i.e.\ $\sr(U)\subseteq\SF^+(U)$. Moreover, $\sr_{loc}(U)\subseteq\SF^+_{loc}(U)$. 
\end{proposition}
\begin{proof}
If $f\in\sr(U)$, then $f\in\mathcal C^n(U,\hh)$. Let $\widetilde f$ be the extension given by Theorem \ref{teo:extension}. Since $\widetilde f$ is a slice function, then it holds $\Gamma^m_K(\widetilde f)=\SP^m_K(\widetilde f)$ for every $m$ and $K$ (Proposition~\ref{pro:1varChar_slice} and \cite[Proposition~9(a)]{CR_S_operators}). 
Therefore for every $y\in U\setminus\R_\bullet$ the function $\Gamma^m_K(f)=\Gamma^m_K(\widetilde f)_{|U}$ is constant on the set $(\s_{y_1}\times\cdots\times\s_{y_{m}}\times\{y_{m+1}\}\times\cdots\times\{y_n\})\cap U $, i.e., $f\in\SF^+(U)$.  
If $f\in\sr_{loc}(U)$, the result follows in the same way using the local extendibility theorem (Theorem \ref{teo:localext}). 
\end{proof}

If $f$ is a slice function on an axially symmetric open set $U=\OO_D$, then it holds $\Gamma^m_K(f)=\SP^m_K(f)$ for every $m$ and $K$. 
Therefore $f$ is strongly slice in $\OO_D$. The converse is a consequence of the following theorem. 

\begin{theorem}[Extendibility of strongly slice functions]\label{teo:slice_extension}
Let $ U \subseteq\hh^n$ be open and let $f\in\mathcal C^0(U,\hh)\cap\mathcal C^n( U\setminus\R_\bullet,\hh)$. Then it holds:
\begin{itemize}
\item[(i)] The function $f$ is strongly slice in $ U $ if and only if it can be extended to a continuous slice function on the symmetric completion $\widetilde U$, of class $\mathcal C^n$ on $\widetilde  U\setminus\R_\bullet$. If this is the case, the extension is unique.
\item[(ii)] The function $f$ is locally strongly slice in $ U $ if and only if  for each $y\in  U $ there exists an open neighborhood $ U _y\subseteq U$ of $y$ and a continuous slice function $\widetilde f^{(y)}$ on $\widetilde{ U _y}$, of class $\mathcal C^n$ on $\widetilde{U_y}\setminus\R_\bullet$, that extends the restriction $f_{| U _y}$.
\end{itemize}
\end{theorem}
\begin{proof}
Let us prove point (i). Assume $f$ strongly slice in $U$. 
For every $K\in\Pm(n)$ and every $y\in U\setminus\R_\bullet$, the functions $\Gamma^{n}_K(f)$ are constant on the set $\s_y\cap U $. Let $\widetilde{\Gamma^{n}_K(f)}$ denote the extension of $\Gamma^{n}_K(f)$ to the whole $\s_y$ obtained by imposing constancy on the sphere and define 
\begin{equation}
\widetilde f(x)=\sum_{K\in\Pm(n)}(-\overline{x})_{K^c}\widetilde{\Gamma^n_K(f)}(x)
\end{equation}
for any $x\in\widetilde U\setminus\R_\bullet$. Since $\widetilde U\cap\R_\bullet=U\cap\R_\bullet$, we can set $\widetilde f(x)=f(x)$ for every $x\in \widetilde U\cap\R_\bullet$. 
Theorem \ref{teo:no_slice_Almansi_several} implies that $\widetilde f=f$ on $U$. 
Since the functions  $\widetilde{\Gamma^{n}_K(f)}$ are circular slice functions in $\widetilde U\setminus\R_\bullet$, also $\widetilde f$ is a slice function in $\widetilde U\setminus\R_\bullet$. 

We show that $\widetilde f\in\mathcal C^n(\widetilde U\setminus\R_\bullet)$. 
If $y\in\widetilde U\setminus\R_\bullet$, let $y'\in\s_y\cap U$ and let $T$ be a rotation of $\hh^n$ around $\R^n$ sending $y$ to $y'$. Let $B_{y'}(r)\subset U\setminus\R_\bullet$ be an open ball. Then $T^{-1}(B_{y'}(r))\subset \widetilde U\setminus\R_\bullet$ is an open neighborhood of $T^{-1}(y')=y$  and it holds
\begin{align}\label{eq:ftildebis}
\widetilde f(x)&=\widetilde f(T^{-1}(x'))=\sum_{K\in\Pm(n)}(-\overline{x})_{K^c}\widetilde{\Gamma^n_K(f)}(T^{-1}(x'))=
\sum_{K\in\Pm(n)}(-\overline{x})_{K^c}\Gamma^n_K(f)(x')=\notag\\
&=
\sum_{K\in\Pm(n)}(-\overline{x})_{K^c}\Gamma^n_K(f)(T(x))
\end{align}
for every $x=T^{-1}(x')$ with $x'\in B_{y'}(r)$.  Therefore $\widetilde f\in\mathcal C^n(T^{-1}(B_{y'}(r)))$.  

If $y\in\widetilde U\cap\R_\bullet$, then there exists an open ball $B_y(r)\subset U$ where $\widetilde f$ coincides with $f$, and then $\widetilde f\in\mathcal C^0(B_y(r))$. Therefore $\widetilde f\in\mathcal C^0(\widetilde U)$, and Remark \ref{rem:extR} implies that $f$ is a slice function in $\widetilde U$. 

Conversely, if $f\in\mathcal C^n(U\setminus\R_\bullet,\hh)$ has a continuous extension $\widetilde f\in\SF(\widetilde U)$ of class $\mathcal C^n$ in $\widetilde U\setminus\R_\bullet$, then  for every $y\in\widetilde U$ the functions $\Gamma^{m}_K(\widetilde f)=\SP^m_K(\widetilde f)$ are constant on the set $(\s_{y_1}\times\cdots\times\s_{y_{m}}\times\{y_{m+1}\}\times\cdots\times\{y_n\})\cap \widetilde U $ (Proposition \ref{pro:SP_slice}).  In view of Proposition \ref{pro:loc_slice_several}, $f$ is strongly slice in $U$.

The uniqueness 
of the extension $\widetilde f$ follows for example from formula \eqref{eq:ftildebis}, that holds locally on $\widetilde U\setminus\R_\bullet$, and from the fact that $\widetilde f(x)=f(x)$ for every $x\in \widetilde U\cap\R_\bullet=U\cap\R_\bullet$. It can be deduced also from the following proposition applied to the extension $\widetilde f$. 

Point (ii) of the statement is an immediate consequence of (i).
\end{proof}

The representation formulas  \cite[Proposition 2.12]{SeveralA} for slice, not necessarily regular, functions, easily implies an identity principle.

\begin{proposition}\label{pro:vanishing}
Let $\OO\subseteq\hh^n$ be an axially symmetric open set and let $U\subseteq\OO$ be an open subset. If $f\in\SF(\OO)\cap\mathcal C^0(\OO,\hh)$ vanishes on $U$, then $f\equiv0$ on $\widetilde U$.
\end{proposition}
\begin{proof}
The one-dimensional case follows easily from the representation formula for slice function \cite[Proposition 6]{AIM2011}. Let $\OO\subseteq\hh$. If $x\in \widetilde U\cap\R=U\cap\R$, then $f(x)=0$. 
If $x=\alpha+I\beta\in\widetilde U\setminus\R\subseteq\OO\setminus\R$ and $\{x_J=\alpha+\beta J, x_K=\alpha+\beta K\}\subset\s_x\cap U$, with $x_J\ne x_K$, then $f(x_J)=f(x_K)=0$, and therefore $f(x)=(I-K)(J-K)^{-1}f(x_J)-(I-J)(J-K)^{-1}f(x_K)=0$. 

If $\OO\subseteq\hh^n$, with $n>1$ and $f\in\SF(\OO)$, $f$ is a slice function w.r.t.\ $x_1$ in $\OO_1(y)$  for every $y\in\OO$, and then the previous argument yields $f\equiv0$ in $\widetilde{U_1(y)}$, where
\[
U_{1}(y):=\{a\in \hh:(a,y_{2},\ldots,y_n)\in U\}.
\]
Therefore $\SP^1_\emptyset(f)=(f)'_{s,x_1}$ and $\SP^1_{\{1\}}(f)=(x_1f)'_{s,x_1}$ vanish identically in $\widetilde U_1\setminus\R_1$, where 
\[
\widetilde U _{1}:=\bigcup_{y\in U}\widetilde{U_1(y)}=\bigcup_{y\in U }\s_{y_1}\times\{y_{2}\}\times\cdots\times\{y_n\}.
\]
Since $f=\SP^1_{\{1\}}(f)-\overline x_1\,\SP^1_\emptyset(f)$ in $\widetilde U_1\setminus\R_1$ and $\widetilde U_1\cap\R_1=U\cap\R_1$, it holds $f\equiv0$ in $\widetilde U_1$. Repeating the same argument on the functions $\SP^1_\emptyset(f)$ and $\SP^1_{\{1\}}(f)$, that are slice functions w.r.t.\ $x_2$, one obtains that $\SP^2_K(f)\equiv0$ on $\widetilde U_2\setminus\R_2$ for every $K\in\Pm(2)$, where
\[
\widetilde U_{2}:=\bigcup_{y\in U }\s_{y_1}\times\s_{y_{2}}\times\{y_3\}\times\cdots\times\{y_n\}.
\]
Theorem \ref{teo:slice_Almansi_several} and continuity gives $f\equiv0$ in $\widetilde U_2$. Proceeding inductively, we get that $f\equiv0$ on every set 
\[
\widetilde U _{m}:=\bigcup_{y\in U }\s_{y_1}\times\cdots\times\s_{y_{m}}\times\{y_{m+1}\}\times\cdots\times\{y_n\}.
\]
Since $ U \subseteq\widetilde U _1\subseteq\cdots\subseteq\widetilde U _n=\widetilde U $, after $n$ steps we get the statement.
\end{proof}

We are now in the position to extend Theorem \ref{teo:Wirtinger_regular} to non axially symmetric sets in $\hh^n$. In particular, we get the converse of Corollary \ref{cor:loc_Wirtinger_regular}.

\begin{theorem}\label{teo:Wirtinger_regular_local}
Let $U\subseteq\hh^n$ be an open set and $f\in\mathcal C^0(U,\hh)\cap\mathcal C^n( U\setminus\R_\bullet,\hh)$. Then it holds:
\begin{itemize}
\item[(i)] If $f$ is strongly slice in $U$, then $f\in\sr(U)$ if and only if $\thb m(f)=0$ for every $m\in\{1,\ldots,n\}$.
\item[(ii)]
If $f$ is locally strongly slice in $U$, then $f\in\sr_{loc}(U)$ if and only if $\thb m(f)=0$ for every $m\in\{1,\ldots,n\}$.
\end{itemize}
\end{theorem}
\begin{proof}
Let $f$ be strongly slice in $U$ with $\thb m(f)=0$ for every $m$. Let $\widetilde f\in\mathcal C^0(\widetilde U,\hh)\cap\mathcal C^n(\widetilde U\setminus\R_\bullet,\hh)$ be the extension given by Theorem \ref{teo:slice_extension}. The functions $\thb m(\widetilde f)$ are continuous slice functions in $\widetilde U\setminus\R_\bullet$, vanishing on $U\setminus\R_\bullet$. Then $\thb m(\widetilde f)=0$ in $\widetilde U\setminus\R_\bullet$ from Proposition \ref{pro:vanishing}. Theorem \ref{teo:Wirtinger_regular} implies that $\widetilde f\in\sr(\widetilde U)$ and then $f\in\sr(U)$. Conversely, if $f\in\sr(U)$, then the extension $\mathcal E(f)=\widetilde f$ provided by Theorem \ref{teo:extension} is slice-regular in $\widetilde U$ and $\thb m(f)=\thb m(\widetilde f)_{|U\setminus\R_\bullet}=0$ for every $m\in\{1,\ldots,n\}$ in view of Theorem \ref{teo:Wirtinger_regular}. 
This proves point (i). Point (ii) is proven in the same way by considering the local extensions $\widetilde f^{(y)}$ of $f$ in place of the global one $\widetilde f$.
\end{proof}

As a consequence of Theorem \ref{teo:Wirtinger_regular_local}, Proposition \ref{pro:loc_slice_several} and Remark \ref{rem:PDE}, we get the following characterization. 

\begin{remark}
Let $U\subseteq\hh^n$ be open.
A function $f\in\mathcal C^0(U,\hh)\cap\mathcal C^{n+1}(U\setminus\R_\bullet,\hh)$ is locally slice-regular if and only if it satisfies the system of partial differential equations
\begin{equation}\label{eq:system1}
\begin{cases}
L^h_{ij}(\Gamma^m_K(f))=0\quad&\forall\, m\in\{1,\ldots,n\},\ 1\le i,j\le 3,\ 1\le h\le m,\ K\in\Pm(m), \\
\thb m(f)=0\quad&\forall\, m\in\{1,\ldots,n\}
\end{cases}
\end{equation}
in $U\setminus\R_\bullet$.
\end{remark}


\end{document}